\documentclass[preprint,10pt]{elsarticle}
\usepackage{amssymb}

\usepackage{amsgen,amsmath,amsxtra,amssymb,amsfonts,amsthm}

\long\def\salta#1{\relax}

\def\gku{G_{k} (u)}

\def\dys{\displaystyle}
\def\vp{\varphi}
\def\al{\alpha}
\def\ga{\gamma}

\newtheorem{theorem}{Theorem}[section]
\newtheorem{lemma}[theorem]{Lemma}
\newtheorem{proposition}[theorem]{Proposition}
\newtheorem{corollary}[theorem]{Corollary}
\newtheorem{remark}[theorem]{Remark}
\newtheorem{remarks}[theorem]{Remarks}
\newtheorem{definition}[theorem]{Definition}

\def\D{\nabla}
\def\intO{\displaystyle \int_{\Omega}}
\def\lio{L^{\infty}(\Omega)}
\def\huz{H^1_0 (\Omega)}
\def\dys{\displaystyle}
\def\rife#1{(\ref{#1})}

\def\eps{\varepsilon}
\newcommand{\elle}[1]{L^{#1}(\Omega)}
\newcommand{\loc}[1]{L^{#1}_{{\rm   loc}}(\Omega)}
\def\vp{\varphi}
\def\la{\lambda}

\newcommand{\na}{\mathbb{N}}

\newcommand{\re}{\mathbb{R}}

\newcommand{\rn}{\mathbb{R}^{N}}

\newcommand{\mis}{\text{meas }}
\newcommand{\diam}{\text{diam}}

\def\defin{\stackrel{\mbox{\tiny def}}{=}}

\def\D{\nabla}
\def\vp{\varphi}
\def\be{\begin{equation}}
\def\ee{\end{equation}}
\def\rife#1{(\ref{#1})}

\def\lp{L^{p}(\Omega)}

\def\t1p0{T^{1,p}_{0}(\Omega)}

\def\m2{M^{\frac{N(p-1)}{N-1}}(\Omega)}

\def\div{\text{div}}

\def\into{\int_{\Omega}}
\def\w-1p'{W^{-1,p'}(\Omega)}

\def\dys{\displaystyle}

\def\lp'n{(L^{p'}(\Omega))^{N}}

\def\vpla{\varphi_{\lambda}}

\journal{}

\begin{document}

\begin{frontmatter}
\title{Existence and nonexistence of solutions for singular quadratic quasilinear equations}

\author[gra]{David Arcoya}
\author[jos]{Jos\'{e} Carmona} 
\author[tom]{Tommaso Leonori} 
\author[gra]{Pedro J. Mart\'{i}nez-Aparicio} 
\author[lui]{Luigi Orsina} 
\author[fra]{Francesco Petitta}
\date{July 21, 2008}
\address[gra]{ Departamento de An\'alisis Matem\'atico, Campus Fuentenueva S/N, Universidad de Granada 18071 - Granada, Spain.  e-mail: darcoya@ugr.es, pedrojma@ugr.es}
\address[jos]{  Departamento de \'Algebra y An\'alisis Matem\'atico, Universidad de Almer\'ia, Ctra. Sacramento s/n
La Ca\~{n}ada de San Urbano 04120 - Almer\'{\i}a, Spain. e-mail: jcarmona@ual.es}
\address[tom]{   CMUC, Departamento de Matem\'{a}tica da Universidade de Coimbra, Largo D. Dinis, Apartado 3008, 3001 -- 454 Coimbra, Portugal. e-mail: leonori@mat.uc.pt}
\address[lui]{Dipartimento di Matematica,  Universit\`a di Roma ``La Sapienza'', P.le Aldo Moro 2, 00185 Roma, Italy. e-mail: orsina@mat.uniroma1.it }
\address[fra]{ 
CMA, University of Oslo, P.O. Box 1053 Blindern, NO-0316 Oslo, Norway. e-mail: francesco.petitta@cma.uio.no}

\begin{abstract}
We study both existence and nonexistence of nonnegative solutions for
nonlinear elliptic problems with singular lower order terms that have
natural growth with respect to the gradient, whose model is
$$
\begin{cases}
\displaystyle -\Delta u + \frac{|\D u|^2}{u^{\gamma}} = f & \mbox{in } \Omega,\\
\displaystyle \hfill u=0 \hfill & \mbox{on } \partial \Omega,
\end{cases}
$$
where $\Omega$ is an open  bounded subset of $\rn $,  $\gamma> 0$
and $f$ is a function which is strictly positive on every compactly
contained subset of $\Omega$. As a consequence of our main results,
we prove that the condition $\gamma<2$ is necessary and sufficient
for the existence of solutions in $\huz$ for every sufficiently
regular $f$ as above.
\end{abstract}


\end{frontmatter}

\section{Introduction}

In this paper we are going to study existence and nonexistence of nonnegative solutions for the following boundary value problem
\begin{equation}   \label{prob}
\begin{cases}
\displaystyle -\mbox{div\,}(M(x,u)\nabla u) + g(x, u) |\nabla u|^2 =
f & \mbox{in } \Omega, \\
\hfill u=0 \hfill &\mbox{on } \partial \Omega.
\end{cases}
\end{equation}
Here $\Omega$ is a bounded, open subset of $\mathbb{R}^N$, $N\geq
3$, $M(x,s)\defin(m_{ij}(x,s))$, $i,j=1,\ldots,N$ is a matrix whose
coefficients $m_{ij}:\Omega\times \mathbb{R} \longrightarrow
\mathbb{R}$ are Ca\-ra\-th\'{e}o\-dory functions (i.e.,
$m_{ij}(\cdot,s)$ is measurable on $\Omega$ for every
$s\in\mathbb{R}$, and $m_{ij}(x,\cdot)$ is continuous on
$\mathbb{R}$ for  a.e. $x\in\Omega$) such that there exist constants
$0 < \alpha \leq \beta$ satisfying
\begin{equation} \label{coe}
\alpha |\varsigma|^2 \leq M(x,s)\varsigma \cdot \varsigma \mbox{ and } |M(x,s)|\leq
\beta, \quad \mbox{ for a.e.
} x\in \Omega, \ \forall (s,\varsigma)\in \mathbb R\times \mathbb R^N.
\end{equation}
The function $g: \Omega \times (0,+\infty) \to \mathbb{R}$ is a Carath\'eodory function (i.e., $g(\cdot,s)$ is measurable on $\Omega$ for every
$s\in (0,+\infty)$, and $g(x,\cdot)$ is continuous on $(0,+\infty)$ for  a.e. $x\in\Omega$) such that
\begin{equation}\label{sign}
g(x,s) \geq 0, \quad \mbox{ for a.e. } x \in \Omega, \ \forall s > 0.
\end{equation}
We will be mainly interested to the case of a function $g$ which is singular
near $s = 0$, such as, for example, $g(x,s) = 1/s^{\gamma}$, $\gamma > 0$. On
the datum $f$, we first suppose that it belongs to $L^{\frac{2N}{N+2}}(\Omega)$
and that it satisfies
\begin{equation}\label{cona}
m_{\omega}(f)\defin\mbox{ess inf\,} \{ f(x) : x\in \omega \}>0,\quad
\forall \omega\subset\subset \Omega.
\end{equation}
Note that \rife{cona} implies that $f \geq 0$ in $\Omega$ and that $f \not\equiv 0$ in $\Omega$.

There are several papers concerned with existence and nonexistence of solutions
for \rife{prob}. If $g$ is nonsingular, that is if  $g$ is a Carath\'{e}odory
function on $\Omega\times [0,\infty)$, problem \rife{prob} has been
exhaustively studied by Boccardo, Murat and Puel \cite{Port}, Bensoussan,
Boccardo and Murat \cite{bbm} and Boccardo, Gallou\"et \cite{bg2} with data $f$
in suitable Lebesgue spaces.

\medskip

    On the contrary, as stated before, in this paper we shall focus our
attention on problem \eqref{prob} with $g(x,s)$ having a singularity at
$s=0$ (uniformly with respect to $x$).
More precisely, we look for a {\sl  distributional solution } of problem \rife{prob},
i.e.\ a function $u\in W^{1,1}_0 (\Omega)$  which solves the equation in the sense of distributions, $u > 0$ almost
everywhere in $\Omega$, and such that$ g(x,u)|\D u|^2$ in $L^1 (\Omega )$. If moreover $u\in\huz$, we say that $u$ is a {\sl finite energy
solution} for problem \rife{prob}. 
A possible motivation for the study of these problems   
arises  from  the  Calculus of Variations.
If $0 \leq f\in \elle{q}$, $q>\frac{N}{2}$ and $\gamma \in (0,1)$,  a purely formal computation shows that the Euler-Lagrange  equation  associated to the functional 
$$
J(v)= \frac12 \intO (1+|v|^{1-\gamma}) |\D v |^2 - \intO f v \, ,
$$
is 
$$
\dys -\div \big( (1+|u|^{1-\gamma}) \D u\big) + \frac{1-\gamma}{2}  \frac{ u}{|u|^{1+\gamma}}|\D u|^2 = f.
$$
Observe that this  is a nonlinear elliptic equation that involves a singular natural--growth gradient term. 

\medskip

 Therefore, it is natural to wonder whether  we can  handle  general not necessarily variational problems whose  simplest model is  
\begin{equation}\label{model}
\begin{cases}
\displaystyle -\Delta u + \frac{|\D u|^2}{u^{\gamma}} = f & \mbox{in } \Omega,\\
\displaystyle \hfill u=0 \hfill & \mbox{on } \partial
\Omega,
\end{cases}
\end{equation}
 and to determine the optimal range of $\gamma >0$ for which solutions exist.

\medskip

Recently,   existence of solutions for \rife{model} has been proved in \cite{ABM,ACM,AM} for
$0<\gamma\leq1$.  We also quote the even more recent papers
\cite{bocc} and \cite{Giachetti-Murat}. Specifically, the existence  of positive solutions of \rife{prob} is proved in \cite{bocc}
provided  $0\not\equiv f\in L^q(\Omega )$ ($q>2N/(N+2)$) with $f\geq0$ and provided
$g(x,s)=1/s^\gamma$ with $\gamma \leq 1$.
On the other hand, a related different problem is studied in
\cite{Giachetti-Murat}. Namely, if  $\mbox{\Large $\chi$}_{\{u>0\}}$
denotes the characteristic function of the set $\{ x\in \Omega :
u(x)>0\}$, $0\leq f\in \lio$, $\mu \in \re$  and $\la ,\gamma >0$,
the differential equation
$$
-\mbox{div\,}(M(x,u)\nabla u) +\la u  + \mu  \frac{|\nabla u|^2}{u^\gamma}
\mbox{\Large $\chi$}_{\{u>0\}}= f
$$
is considered. The given results about existence of nonnegative
solutions in $H_0^1(\Omega)$ depend on $\gamma$. Indeed, existence
is proved for every $\mu \in \re$ if $\gamma<1$, while the case
$\gamma\geq 1$ requires that $\mu <0$. Thus, if $\gamma\geq 1$ the
term with quadratic dependence in $\nabla u$ is negative (i.e., the
opposite assumption with respect to \rife{sign}). 
In this direction, result for similar equations can be also found in \cite{gp} and \cite{pv} (see also references cited therein).

\vspace{0.5cm}

The purpose of this paper is twofold. First of all, we will extend the above
results to a more general class of nonlinearities both in the principal part of
the operator and in the  lower order term, as well as to general, possibly
$L^1(\Omega)$, data. Then, we will give a sharp range of nonlinearities
$g(x,s)$ for which these problems  admit a solution for every datum $f\in
L^q(\Omega )$, with $q>N/2$,  satisfying \eqref{cona}.

{In order to prove our results, we will have to strengthen assumption
\rife{sign}. Specifically, for the results of existence of solutions, we will
suppose that the function $g(x,s)$ satisfies }
\begin{equation} \label{cong-2}
0\leq g(x,s)\leq h(s), \quad \mbox{for a.e. } x\in \Omega,\
\forall s>0,
\end{equation}
where $h : (0,+\infty) \to [0,+\infty)$ is a continuous nonnegative function
such that
\begin{equation}\label{cong-1}
\begin{array}{c}
\dys
\lim_{s \to 0^{+}}\,\int_{s}^{1}\,\sqrt{h(t)}\,dt < +\infty,
\\ h(s) \mbox{ is nonincreasing in a neighborhood of zero.}
\end{array}
\end{equation}
Our result of existence of finite energy solutions (proved in Section~2) is the following.

\begin{theorem} \label{teorema}\sl
Let $f$ in $L^{\frac{2N}{N+2}}(\Omega)$ be such that \rife{cona}
holds, and suppose that \rife{coe}, (\ref{cong-2}) and
(\ref{cong-1}) hold. Then there exists a finite energy solution $u$
for problem (\ref{prob}). Furthermore, 
$u \,
g(x,u)|\D u|^2 \in L^1 (\Omega)$.
\end{theorem}

Note that the fact  $u \, g(x,u)|\D u|^2 \in L^1 (\Omega)$ implies
that the solution $u$ itself is allowed as test function
(since $f
\in H^{-1} (\Omega)$) in the weak formulation of \rife{prob} (see \rife{enerfin} in Section~2). With respect to the proof, due to the fact
that the lower order term $g(x, u) |\nabla u|^2$ is (possibly)
singular as the solution is near $0$, we will approximate the
function $g(x,s)$ by nonsingular ones $g_n(x,s)$ in such a way that
the corresponding approximated problems have finite energy solutions
$u_n$ for every $n$ in $\mathbb{N}$. The main difficulty in the
proof of Theorem \ref{teorema} relies on a suitable local uniform
estimate from below of these solutions. To do it,
it suffices by \rife{cong-2} to prove that any   supersolution $z>0$ for the
equation
$$
- \mbox{{\rm div\,}} (M(x,z) \nabla z) + h(z) |\nabla z|^2 = f \quad \mbox{in}\,\, \Omega\,
$$
is above some positive {constant  in} every $\omega\subset\subset\Omega$, i.e.
\begin{equation}\label{nabo}
{\forall \omega\subset \subset \Omega\quad\exists c_\omega>0:}\quad z(x)\geq
c_{\omega}>0 .
\end{equation}
This is proved in Proposition~\ref{prop2}
via a suitable change of variable
which turns the goal into a local $L^{\infty}$ estimate for solutions of
quasilinear problems. The  local $L^{\infty}$ estimate is then obtained using
a result of \cite{LEO} (see also the pioneering paper  \cite{Brez} and also
\cite{BoccGallVaz,GalMor}) on an equation whose model is
\begin{equation}\label{semil}
-\div (\widetilde{M}(x,v) \D v ) + f(x) b (v)=0\quad \mbox{in }
\Omega,
\end{equation}
where $\widetilde{M}$ satisfies \rife{coe} and { $b(s)$ is a function with $b(s)/s$ increasing for large $s>0$ and satisfying the Keller-Osserman condition
$$
\int^{+\infty} \frac{dt}{\sqrt{{2}\int_0^t b(\tau)d \tau} }<+\infty .
$$
}
For the convenience of the reader, the exact result that we need is
proved in the appendix (see Theorem~\ref{thleoni}). For such type of
$L^{\infty}$ estimates we refer to the ``classical''   literature on
the so-called large solutions (see, among others,
\cite{BaM,MV,MV2,Ve}) and on local estimates (see, among others, {\cite{BoccGallVaz,Brez,GalMor,LEO,jlv}}).

\smallskip

Section \ref{fur} of this paper will be concerned with some extensions of the existence result. First of all, combining the above ideas with those in \cite{Porretta} (see also
\cite{leoporr}), we handle the case of data $f$ in $\elle{1}$, proving the existence of
distributional solutions $u$ of \rife{prob}, with $u$ in $W^{1,q}_0
(\Omega)$ for every $q<\frac{N}{N-1}$. More precisely, in Section \ref{sec3}, we shall prove the
following result.

\begin{theorem}\label{th1}\sl
Let $f$ in $L^{1}(\Omega)$ be such that \rife{cona} holds and suppose that \rife{coe}, (\ref{cong-2}) and (\ref{cong-1}) hold. Then
there exists a distributional solution $u$ of \rife{prob}, with $u$ in $W^{1,q}_0 (\Omega)$, for every $ q<\frac{N}{N-1}$.
If, in addition,  there exist $s_0 > 0$ and $\mu >0$ such that
\begin{equation} \label{coerc}
g(x, s)\geq \mu \quad \mbox{for a.e. } x \in \Omega, \quad \forall  s\geq s_0,
\end{equation}
then  $u\in H^1_0 (\Omega)$ (i.e., it is a finite energy solution).
\end{theorem}

On the other hand, in Section \ref{nodi}, we will also provide an
analogous of Theorem  \ref{teorema} involving more general
differential operators whose principal part is not in divergence
form and data in $\elle{q}$ with $q>\frac N2$. Namely, we consider
the following problem
\be\label{nodive}
\begin{cases}
\displaystyle -\sum_{i,j=1}^N  a_{ij}(x) \frac{\partial^2
u}{\partial x_i\partial x_j}(x) + \sum_{i=1}^N b_i (x)
\frac{\partial u}{\partial x_i}(x) + g(x, u) |\nabla u|^2 =
f & \mbox{in } \Omega, \\
\hfill u=0 \hfill &\mbox{on } \partial \Omega,
\end{cases}
\ee where the coefficients $a_{ij} (x)$ satisfy the ellipticity
condition \be\label{ell} 0<\alpha |\varsigma|^2 \leq \sum_{i,j=1}^N
a_{ij}(x) \varsigma_i \varsigma_j \leq \beta |\varsigma|^2, \quad
\forall \varsigma \in \rn\,, \ee for some $0<\alpha \leq \beta$. We
prove the following result.

\begin{theorem}\label{notdiv}
Suppose that   $a_{ij} \in W^{1,\infty} (\Omega)$ satisfy
\rife{ell}, and that $b_i \in \lio$. Assume that $f(x)$ satisfies
\rife{cona} and belongs to $\elle{q}$  with $q>\frac{N}{2}$. Suppose
moreover that $g(x,s)$ satisfies \rife{sign}, \rife{cong-2} (with
$h$ such that \rife{cong-1} holds). Then  there exists a solution
$u\in \huz\cap\lio$ for \rife{nodive}. Furthermore, $g(x,u)|\D
u|^2\in L^1(\Omega)$.
\end{theorem}

\smallskip

We are also concerned with nonexistence of positive solutions for problem \rife{prob} for data
$f$ in $L^q(\Omega)$ for some $q>\frac N2$, with $f \geq 0$ and $f\not\equiv 0$.
In contrast with the previous existence results, we will assume in this case
that the nonlinearity $g(x,s)$ is above a function $h(s)$ whose square root is
not integrable in $(0,1)$. Specifically, we assume that
\begin{equation}    \label{hcero}
0\leq h(s) \leq g(x,s), \quad  \mbox{for a.e. } x\in \Omega, \ \forall s>0,
\end{equation}
where $h:(0,+\infty)\to [0,+\infty)$ is a nonnegative continuous function  such
that
\begin{equation}\label{contador}
\lim_{s \to 0^{+}}\,h(s) = +\infty,
\quad
\lim_{s \to 0^{+}}\,\int_{s}^{1}\,\sqrt{h(t)}\,dt = +\infty,
\end{equation}
and
\begin{equation}\label{pereiro}
\displaystyle \lim_{s\to 0^+} \sqrt{h(s)}\  {\rm e}^{\displaystyle
\int_1^s \sqrt{h(t)}dt} = h_0\geq 0.
\end{equation}
Among  others, we are going to prove in Section~4 that if
$\lambda_1(f)$ denotes the first positive eigenvalue of the
laplacian operator $-\Delta$ with zero Dirichlet boundary conditions
and weight $f\in L^q(\Omega )$, ($q>N/2$), then the following result
holds.

\begin{theorem} \label{teoremah}\sl
Let $f$ in $L^q(\Omega )$, with $q>\frac{N}{2}$, be such that $f \geq 0$ and $f\not\equiv 0$, and assume that \rife{coe}, \rife{hcero},
\rife{contador}, \rife{pereiro} hold. If $\lambda_1(f)>
\frac{\beta}{\alpha}$, then \rife{prob} does not have any finite energy solution.
\end{theorem}

As an easy consequence of Theorem \ref{teoremah}, we will prove (see
Corollary~\ref{casooptimo}) that the model problem \rife{model} does
not have any finite energy solution  provided $\gamma\geq 2$. By
gathering together  this nonexistence result  and
Theorem~\ref{teorema} we conclude immediately that, in the case of
the model problem \rife{model}, we have a sharp range of values
of $\gamma$ for which there exist solutions. In addition, if
$\gamma$ is not in this range, we prove also  what happens if we try to approximate problem \rife{model} with a sequence of problems for which solutions exist.

\begin{theorem}\label{modelt}\sl
Problem $\rife{model}$ has a finite energy solution for every $f\in
\elle{q}$ ($q > \frac N2$) satisfying \rife{cona} if and only if
$\gamma<2$. Moreover, let $\lambda_{1}$ be the first eigenvalue of
the laplacian in the $N$-dimensional unit ball (i.e.  the first
positive zero of the Bessel function $J_m$ with $m=N/2-1$), assume
$f\in \lio$, and either \be\label{ipnnex} \ga >2 \quad \mbox{ or }
\quad \ga=2 \mbox{ and } \|f\|_{\elle{\infty}}<
\frac{\lambda_{1}}{\diam{\left(\Omega\right)}^2}.
\ee
Then the sequence $\{u_n\}$ of solutions of
$$
\begin{cases}
\dys -\Delta u_n +\frac{|\D u_n |^2}{ \left(u_n + \frac{1}{n}\right)^{\gamma} } = f & \mbox{in } \Omega, \\
\hfill u_n = 0 \hfill & \mbox{on } \partial \Omega,
\end{cases}
$$
tends to $0$ in $\huz$, and the sequence $\frac{|\D u_n |^2}{ \left(u_n + \frac{1}{n}\right)^{\gamma} }$ converges to $f$ in the weak$*$ topology of measures.
\end{theorem}

To conclude this introduction, some remarks are in
order. First, we have to mention that uniqueness of solutions for
$\eqref{model}$ is proved in \cite{AS} for the case $0<\gamma <
1$. Secondly,
let us explicitly state that we have chosen to  present the results and to perform
 the proofs  in the case $N\geq 3$. However,   all the results but Theorem \ref{teorema}
 hold true also in the case $N=2$ (with easier proofs). In addition,
if $N=2$ (which implies $\frac{2N}{N+2}=1$), Theorem \ref{teorema}
is also true provided we replace the assumption $f\in
L^{\frac{2N}{N+2}}(\Omega)$ with $f\in L^{m}(\Omega)$, and assume
$m>1$.

\medskip

The plan of the paper is the following: in Section 2 we will prove
{a} local estimate from below for the solutions, together with
Theorem \ref{teorema}. Section 3 is devoted to provide further
existence results for $L^1$ data (Theorem \ref{th1}) and operators
in non divergence form (Theorem \ref{notdiv}). In Section 4 we prove
the nonexistence result (both theorems \ref{teoremah} and
\ref{modelt}). Finally we present in the Appendix some results
related to the local estimate \rife{nabo}. For instance, we show in
detail how to get the lower bound for solutions of \rife{prob},
{through} a suitable change of variable, proving a local bound from
above {for solutions} of a semilinear equation whose model is
\rife{semil} {(Theorem~\ref{thleoni})}. Such topic is strictly
related to the possibility of {constructing} estimates for solutions
of \rife{semil} that do not depend on the behavior at the boundary:
and indeed in Theorem \ref{existls} we prove the existence of
solutions that blow-up at the boundary (i.e., the so-called \lq\lq
large solutions\rq\rq ) for such equations.

\medskip

\noindent {\bf Notation}.  For any $k>0$ we set $
T_k (s) =\min ( k , \max ( s, -k )) $ and $G_k (s)= s- T_k (s)$.
Moreover, for any $q>1$, $q' = \frac{q}{q-1}$ will be the
H\"{o}lder conjugate exponent of $q$, while for any $1<p<N$, $p^*=
\frac{Np}{N-p}$ is the Sobolev conjugate exponent of $p$. As usual,
$\mathcal{S}$ denotes the best Sobolev constant, i.e.,
$$
\mathcal{S}=\sup \{ \| u\|_{L^{2^*} (\Omega)} : \| u\|_{H_0^{1} (\Omega)}=1 \}.
$$

In Section~3 we will use some ideas related to  Marcinkiewicz spaces; for the
convenience of the reader we recall  here their definition and some properties.
For $s>1$, we denote by $\mathcal M^{s}(\Omega)$  the space of measurable
functions $v:\Omega \to \re $ such that there exists $c>0$, with
\begin{equation}\label{ma}
\mbox{meas}\{x\in \Omega : |v(x)|\geq k\}\leq \frac{c}{k^s }, \quad \forall k>0 .
\end{equation}
The space $\mathcal M^s (\Omega)$ is a Banach space, and on it can be defined the pseudo-norm
$$
\|v\|^s_{\mathcal M^{s}(\Omega)}=\inf \left\{ c>0: \mbox{ \rife{ma} holds}\right\}.
$$
We also recall that, since $\Omega$  is bounded, for every $\eps \in (0,s-1]$,
there exists a positive constant $C$ such that
\begin{equation}\label{imm}
\begin{array}{c}
\|v\|_{\mathcal M^{s}(\Omega)} \leq {\,\|v\|_{L^{s}(\Omega)}},\quad \forall v
\in L^{s}(\Omega) \\
\|{w}\|_{L^{s-\eps}(\Omega)} \leq
C\,\|{w}\|_{\mathcal M^{s}(\Omega)} ,\quad
\forall {w}\in  \mathcal M^{s}(\Omega) .
\end{array}
\end{equation}

Finally, following \cite{Port}, we set $\vp_{\la } (s)= s e^{\la s^2}$, $\la
>0$; in what follows we will use that for every $a$, $b >0$ we have
\begin{equation}\label{vpla}
a\vp_{\la }' (s)- b |\vp_{\la } (s)| \geq \frac a2,
\end{equation}
if $\la >\frac{b^2}{4 a^2}$. We will also denote by $\eps (n)$ any
quantity that tends to $0$ as $n$ diverges.

\vspace{0.5cm} \noindent {\bf Acknowledgment}: The authors wish to
express his thanks to Prof. Serrin for his suggestion of extending
our existence main result to non-divergence operators.
\setcounter{equation}{0}
\section{Finite energy solutions}\label{sec2}

In this section we will prove the existence of finite energy
solutions for problem \eqref{prob}. Let us recall its definition.

\begin{definition}\rm
A {\sl supersolution (resp.\ subsolution)} for problem \rife{prob} is a function $u\in
W^{1,1}_{\rm loc} (\Omega)$ such that
\begin{itemize}
\item[1)] $u > 0$ almost everywhere in $\Omega$,

\item[2)] $ g(x,u)|\D u|^2$ belongs to $\loc{1}$,

\item[3)] for every $0\leq \phi \in C_c^{\infty} (\Omega)$, it holds
$$
\intO M(x,u)\D u \cdot \D \phi + \intO g (x,u)|\D u|^2 \phi \underset{(\leq)}\geq  \intO f  \, \phi
\, .
$$
\end{itemize}
A function $u\in W^{1,1}_0 (\Omega)$ is a  {\sl distributional solution} for  \rife{prob}
if $g(x,u)|\D u|^2$ belongs to $L^1(\Omega)$, and $u$ is both a supersolution and a subsolution for such a problem.\\
\noindent If moreover $u\in\huz$, we say that $u$ is a {\sl finite
energy solution} for problem \rife{prob}. In this case, we have
\begin{equation}\label{enerfin}
\intO M(x,u)\D u \cdot \D \psi + \intO g (x,u)|\D u|^2 \psi = \intO f  \,
\psi, \quad \forall \psi \in \huz\cap\lio.
\end{equation}
\end{definition}

 The
proof of Theorem~\ref{teorema} relies on approximating the datum $f\in L^{\frac{2N}{N+2}} (\Omega)$ by
its truncature $f_n=T_n(f)$ and the nonlinearity $g$ by a  suitable
sequence of Carath\'{e}odory functions
$g_n$ (for $n\in
\mathbb{N}$). Specifically, we define
$$
g_n(x,s)\defin\left\{
\begin{array}{cl}
g(x,s)& \displaystyle \quad s\geq \frac{1}{n},
\\
\displaystyle nh\left( \frac{1}{n}\right) \frac{s}{h(s)}g(x,s)& \displaystyle  \quad 0< s\leq \frac{1}{n},
\\
0& \displaystyle  \quad s\leq 0.
\end{array}
\right.
$$
Since $h$ is nonincreasing in a neighborhood of zero, we observe that there
exists $n_0\in \na$, such that $g_n$ satisfies, for a.e. $x\in \Omega$,
$\forall s> 0$,
\begin{equation}\label{condult}
\begin{cases}
\displaystyle
\hfill \lim_{n \to +\infty}\,g_n(x,s) = g(x,s)\,, \hfill \\[1.5 ex]
g_n(x,s)\leq g(x,s) \,, \forall n\geq n_0\,,\\[1.5 ex]
\hfill g_n(x,s) \geq 0\,. \hfill
\end{cases}
\end{equation}

Since for fixed $n$  both functions $f_n(x)$ ($x\in\Omega $) and
$\frac{|\varsigma|^2}{1+\frac{1}{n}|\varsigma|^2}$  ($\varsigma\in\mathbb R^N$) are bounded,
classical results allow us to
deduce that problem
\begin{equation}    \label{(P_n)}
\begin{cases}
  \displaystyle  -\mbox{div\,}(M(x,u_n)\nabla u_n) +g_n(x,u_n )\frac{|\nabla u_n|^2}{1+\frac{1}{n} |\nabla u_n|^2}
       =f_n & \mbox{in } \Omega, \\
   \hfill u_n={0} \hfill & \mbox{on } \partial \Omega,
\end{cases}\end{equation}
has a solution $u_n$ that belongs to $ H_0^1(\Omega)$ (see \cite{lelio}) and to $\lio$ (see \cite{S}).

We are going to prove now some properties of the sequence $u_n$ that
we will use in the sequel.

\begin{lemma}  \label{lemaprop}\sl
Assume that $0\not\equiv f\in L^{\frac{2N}{N+2}}(\Omega)$ satisfies $f\geq 0$
and that $M(x,s)$ satisfies \rife{coe}. If, for every $n\in\mathbb N$, the
function $u_n\in H_0^1(\Omega)$ is a solution of problem \rife{(P_n)}, then:
\begin{enumerate}
\item[1.] The sequence $ \{ u_n \} $ is bounded in ${H_0^1(\Omega)}$  and
$$
u_n g_n
(x,u_n)\frac{|\nabla u_n|^2}{1+\frac{1}{n} |\nabla u_n|^2} \mbox{ is bounded in } L^1(\Omega ).
$$

\item[2.] The functions $ u_n$  are continuous  in $\Omega$ and $u_n(x)>0$ for every $x\in \Omega$ and $n\in \mathbb N$.
\end{enumerate}

\end{lemma}
\begin{proof}[{\sl Proof}] $1$. Taking $  u_n$ as test
function in $\rife{(P_n)}$ and using H\"{o}lder
and Sobolev inequalities we obtain that
$$
  \int_\Omega M(x,u_n)\nabla u_n \cdot\nabla u_n \,
     +
   \int_\Omega g_n(x,u_n)u_n\frac{|\nabla u_n|^2}{1+\frac{1}{n} |\nabla u_n|^2}
   =\int_\Omega f_n u_n
$$
$$
\leq \mathcal{S} \| f\|_{L^{\frac{2N}{N+2}} (\Omega)} \|\D
  u_n\|_{L^2 (\Omega)}.
$$
By the ellipticity condition \rife{coe} and the nonnegativeness of $ g_n(x,s)
s$, we conclude that
the sequences $ u_n$ and $ \dys 
u_n g_n (x,u_n)\frac{|\nabla u_n|^2}{1+\frac{1}{n} |\nabla u_n|^2}
$ are  bounded, respectively,  in $H_0^1(\Omega)$ and in $L^1(\Omega )$.

2. We  take $u_n^-\defin \min ( u_n,0)$ as test function in $\rife{(P_n)}$, so that, by \rife{coe},
$$
\alpha   \int_\Omega |\nabla
u_n^-|^2 \,
      +
\int_\Omega g_n(x,u_n) \frac{|\nabla u_n|^2}{1+\frac 1n |\nabla u_n|^2} u_n^-
\, \leq \int_\Omega f_n\,u_n^- \,.
$$
Using that $f_n\geq 0$ and $g_n(x,s)$ is zero for every $s\leq 0$, we obtain
$$
\alpha \int_\Omega |\nabla u_n^-|^2 \,
             \leq
\int_\Omega f_n\,u_n^- \, \leq 0.
$$
Thus {$u_n^-\equiv 0$ and so} $u_n\geq0$. \salta{Moreover, choosing $ G_k (u_n)$ as
test function in \rife{(P_n)}, we deduce from \rife{coe} and since
{$g_n(x,u_n)\geq 0$}
\begin{equation}    \label{Stp}
\alpha\intO |\D G_k (u_n)|^2 \, \leq \intO f_n \,G_k (u_n) .
\end{equation}
Thanks to a result by Stampacchia, we deduce that  any
$ u_n $ belongs to $ L^\infty (\Omega )$ and }
Moreover,
\begin{equation*}
      \begin{array}{c}
     \displaystyle  -\mbox{div\,}(M(x,u_n)\nabla u_n)=f_n-g_n(x,u_n)\frac{|\nabla u_n|^2}{1+\frac{1}{n} |\nabla
     u_n|^2}\in L^{\infty}(\Omega).
                \end{array}
\end{equation*}
Hence  $u_n$ belongs to the space of the H\"older continuous
functions in $\Omega$  (see for instance \cite{lad}, Theorem~1.1 in
Chapter~4).

We are now going to prove that $u_n>0$ in $\Omega$. Let $C_n>0$ be
such that $g_n(x,s) \leq C_n s$, for $s\in
\left[0,\|u_n\|_{L^\infty(\Omega)} \right]$. Thus the nonnegative
function $u_n$ satisfies in the sense of distributions in $\Omega$
\begin{eqnarray*}
-\mbox{div\,}(M(x,u_n)\nabla u_n) + n C_n u_n &\geq&
-\mbox{div\,}(M(x,u_n)\nabla u_n) + \frac{  g_n(x,u_n) |\nabla u_n |^2}{1+\frac
1n |\nabla u_n|^2} \\
&=& f_n  \, .
\end{eqnarray*}
Observing that $f_n$ is nonnegative and not identically zero (since
{$f\not\equiv 0$}), by the strong maximum principle (see \cite{GT}
for instance) we deduce that  $u_n>0$ in $ \Omega$.
\end{proof}

In the next proposition we will prove  that the sequence  $\{u_n\}$  is uniformly bounded from below, away from zero, in every compact set in $\Omega$.
This result will be crucial in order to prove the existence of a solution for \rife{prob}.

\begin{proposition} \label{prop2}\sl
Suppose that $f \in L_{\mbox{\tiny loc}}^{\infty}(\Omega )$ satisfies
\rife{cona}, and that  $h$ is such that
\rife{cong-1} holds. Let $\omega$ be a compactly contained open subset of
$\Omega$. Then there exists a constant $c_{\omega}>0$ such that every
supersolution  $0<z\in H_{\mbox{{\tiny {\rm loc}}}}^1(\Omega )\cap C(\Omega )$
of the equation
\begin{equation}
                                           \label{eqh}
- \mbox{{\rm div\,}} (M(x,z) \nabla z) + h(z) |\nabla z|^2 = f
\quad \mbox{in } \Omega,
\end{equation}
satisfies
$$
z \geq c_{\omega} \quad \mbox{in }\omega.
$$
\end{proposition}

\begin{remark} \label{recta}
{\rm The above proposition will be crucial in the proofs of both
Theorem~\ref{teorema} and \ref{th1}. In fact, we will use the following
consequences:
\begin{enumerate}
\item[(i)]  Let  $u_n$ be a solution of \rife{(P_n)} with $n\geq n_0$ ($n_0$ given by
\eqref{condult}).
By Lemma~\ref{lemaprop}, $u_n>0$ in $\Omega$ and it is continuous.
In particular $h(u_n)|\nabla u_n|^2\in
\loc{1}$. Thus, from 
the inequalities $g_n(x,s)\leq
g(x,s)\leq h(s)$ for every $s>0$ and $f_n\geq f_1$ 
we obtain that
$u_n$ is a supersolution for
\begin{equation*}
- \mbox{{\rm div\,}} (M(x,z) \nabla z) + h(z) |\nabla z|^2 = f_1 \quad \mbox{in } \Omega .
\end{equation*}
Therefore, by the above proposition (with $f=f_1$ and ${z=u_n}\in H_0^1(\Omega
) \cap C(\overline{\Omega})$ (Lemma~\ref{lemaprop}-2.)) for any
$\omega\subset\subset \Omega$ we get the existence of a positive
constant $c_{\omega}$ such that $u_n \geq c_{\omega}$ in $\omega$. Taking $k>0$
and $m_0>\max \{ n_0, \frac{1}{c_\omega}\}$, we deduce, by the
definition of $g_n$, that for all $n\geq m_0$
$$
g_n(x, u_n(x)) = g(x,u_n(x)) \leq c_k(\omega)\defin\max_{s\in [c_\omega, k]} h(s)
\,,
$$
for every $x\in \omega$ such that $u_n(x)\leq k$.

\vspace{0.5cm}
\item[(ii)] If $0<u_n\in H_{0}^1(\Omega )\cap C(\Omega )$ is a finite energy solution of
$$
- \mbox{{\rm div\,}} (M(x,{u_n}) \nabla {u_n}) + g(x,{u_n}) |\nabla {u_n}|^2 =
f_n \quad \mbox{in } \Omega ,
$$
then, using again that $g(x,s)\leq h(s)$, $f_n\geq f_1$
and $h(u_n)|\nabla u_n|^2\in \loc{1}$, we derive that $u_n$ is also
a supersolution of
\begin{equation*}
- \mbox{{\rm div\,}} (M(x,z) \nabla z) + h(z) |\nabla z|^2 = f_1 \quad \mbox{in } \Omega .
\end{equation*}
Consequently, if $\omega\subset\subset \Omega$ and $c_\omega$ has been defined
above (with $f=f_1$), then  $u_n \geq c_{\omega}$ in  $\omega$. Therefore,
$$
g(x,u_n(x) ) \leq c_k(\omega)\defin\max_{s\in [c_\omega, k]} h(s) \,,
$$
for every $x\in \omega$ such that $u_n(x)\leq k$.
\end{enumerate}
}
\end{remark}

\begin{proof}[{\sl Proof of Proposition \ref{prop2}}]   Let $z>0$ be a
supersolution  of \rife{eqh}. We are going to consider a suitable change of
variable. In order to make it, since in general the function $h$ may be
integrable in $(0,1)$, we set $
\dys \widetilde{h}(s)= h(s) +\frac{\alpha}{s}
$, and define, for $s > 0$, the nondecreasing function
\be\label{acca}
H(s)=\int_1^s \widetilde{h}(t)dt=  \int_1^s {h}(t)dt + \log \,
s^{\alpha},
\ee
and the nonincreasing function
\be\label{change}
\psi(s)=\int_s^1 {\rm e}^{-\frac{H(t)}{\alpha}}dt
=\int_s^1 t^{- 1}{\rm e}^{-\frac{\int_1^t h(\tau )d\tau}{\alpha}} dt.
\ee
Observing that
$$\lim_{s \to 0^{+}}\,\psi(s)=+\infty,
\qquad \lim_{s \to +\infty}\,\psi(s){=\psi_\infty}\in[-\infty,0),
$$
we can define
\be\label{change2}
v\defin \psi (z).
\ee

Since $z$ is continuous and strictly positive
in $\Omega$, we get that $z$ is bounded away from zero (with the
bound depending on $z$) in every open set $\omega$ compactly
contained in $\Omega$. Consequently, by the chain rule, we have
\begin{equation}\label{chain}
\nabla v=  -{\rm e}^{-\frac{H(z)}{\alpha}} \nabla z \in L^2(\omega),\quad
\forall \omega\subset\subset\Omega,
\end{equation}
and thus $v\in H^1 (\omega)$ for every $\omega \subset\subset \Omega$, i.e.,
$v\in H^1_{\mbox{\tiny loc}} (\Omega)$.

Let  $0\leq \phi\in
C_c^{\infty}(\Omega)$, and take (as in \cite{bocc}) ${\rm e}^{-\frac{H(z)}{\alpha}}\phi$ as test
function in \rife{eqh} to deduce from the inequality $ h(s)\leq
\widetilde{h}(s)$ that
\begin{eqnarray*}
- \int_{\Omega}M(x,z)\nabla z \cdot\nabla z \, \frac{\widetilde{h}(z)}{\alpha}
{\rm e}^{-\frac{H(z)}{\alpha}}\phi \, +\int_{\Omega}M(x,z)\nabla z \cdot \nabla
\phi\, {\rm e}^{-\frac{H(z)}{\alpha}} \,
\\
+ \int_{\Omega} \widetilde{h}(z)|\nabla z|^2{\rm
e}^{-\frac{H(z)}{\alpha}}\phi\,
\geq \int_{\Omega}f
{\rm e}^{-\frac{H(z)}{\alpha}} \phi\,.
\end{eqnarray*}
Using \rife{coe} together with \rife{chain} we get,
\begin{eqnarray*}
-\int_{\Omega}M(x,z)\nabla \psi(z) \cdot \nabla \phi \,
            \geq
\int_{\Omega}f {\rm e}^{-\frac{H(z)}{\alpha}} \phi\,
 { \geq
\int_{\Omega}  \left( {\rm e}^{-\frac{H(z)}{\alpha}}  -1\right) f \phi\,
 }.
\end{eqnarray*}

\noindent
If we define $\widetilde{M}(x,s)=
M(x,\psi^{-1} (s))$ and

\be\label{b(s)}
{b(s)= {\rm e}^{-\frac{H(\psi^{-1}(s))}{\alpha}} -1}\quad \mbox{ for every
 } s\in(\psi_\infty,+\infty),
 \ee
 then $v$ is subsolution of
\begin{equation*}   
-\mbox{div\,}(\widetilde{M}(x,v)\nabla v)+ f (x) \, {b(v)}=0 \quad \mbox{in}\,\,
\Omega.
\end{equation*}

Observe that $\frac{b(s)}{s}$ is nondecreasing for large $s>0$; indeed, this is
equivalent to prove that $ \Upsilon (t)= \frac{{\rm e}^{-\frac{H(t)}{\alpha}}-1}{\psi
(t)}$ is nonincreasing in a neighborhood of $t=0$. To show this, let $w_0\in
(0,1)$ be such that $\widetilde{h}(t)$ is nonincreasing in $(0,w_0]$, and, note that
$$
-{\rm e}^{\frac{H(t)}{\alpha}}\psi^2 (t)\Upsilon '(t)= \frac{\widetilde{h}(t)}{\alpha}\psi (t)-({\rm e}^{-\frac{H(t)}{\alpha}}-1)=\int_t^1\frac{ [\widetilde{h}(t)-\widetilde{h}(s)]}{\alpha}{\rm e}^{-\frac{H(s)}{\alpha}} ds
$$
$$
\geq \int_{w_0}^1 \frac{[\widetilde{h}(t)-\widetilde{h}(s)]}{\alpha} {\rm e}^{-\frac{H(s)}{\alpha}} ds = \widetilde{h}(t)M_1 -M_2
$$
where
$$
M_1= \frac{1}{\alpha} \int_{w_0}^1 {\rm e}^{-\frac{H(s)}{\alpha}} ds \quad \mbox{and}
\quad
M_2=\frac{1}{\alpha} \int_{w_0}^1 \widetilde{h} (s) {\rm e}^{-\frac{H(s)}{\alpha}}
ds. $$
Thus, if $t$ belongs to the interval $(0,\widetilde{h}^{-1}(\min\{ w_0 ,
M_2/M_1\}) )$, then the right hand side of the above inequality is positive,
and consequently $\Upsilon (t)$ is nonincreasing in this interval.

We also claim now that  since $\int_0^1\sqrt{h(s)}ds<+\infty$ and $h$ is nonincreasing in a
neighborhood of zero, then  the function $b(s)$ satisfies the well-known
Keller-Osserman  condition (see \cite{K} and \cite{o} for instance), i.e.,
there exists $t_0>0$ such that
\begin{equation}                     \label{k-o}
\int_{t_0}^{+\infty}\frac{dt}{ \sqrt{2 \int_0^t b(s) ds }} <
+\infty .
\end{equation}
We postpone the proof of the claim for the moment, and we show how to
conclude the proof by using the claim. Indeed,
by applying \cite[Theorem 7]{LEO} (see also Theorem~\ref{thleoni}
in the Appendix where, for the convenience of the reader, we have
also included a proof of the precise result that we need here)
we derive that for every $\omega \subset\subset \Omega$, there exists
$C_{\omega}>0$ such that
$$
v\leq C_{\omega} \quad \mbox{in } \omega .
$$
Therefore, undoing the change
$$
z\geq\psi^{-1}(C_{\omega})=  c_{\omega}>0   \quad  \mbox{in } \omega\,,
$$
as desired.

Consequently, to conclude the proof it suffices to show \rife{k-o} or, equivalently, that
$$
\int_{t_0}^{+\infty}\frac{dt}{ \sqrt{2 \int_0^t {\rm e}^{-\frac{H(\psi^{-1}(s))}{\alpha}} ds }} <
+\infty .
$$
\noindent
Using  the change  $\tau =\psi^{-1}(s)$,
we obtain

\begin{eqnarray*}
\displaystyle \int_{t_0}^{+\infty}\frac{dt}{\sqrt{2\int_0^t {\rm e}^{-\frac{H(\psi^{-1}(s))}{\alpha}}
ds }} &=&
\int_{t_0}^{+\infty}\frac{dt}{\sqrt{2\int_{\psi^{-1}(t)}^{
 \psi^{-1}(0)}
{\rm e}^{-2\frac{H(\tau )}{\alpha}} d\tau}}.
\end{eqnarray*}
Now we apply  the change $w=
\psi^{-1}(t)$ to  deduce that
\begin{eqnarray*}
\int_{t_0}^{+\infty}\frac{dt}{\sqrt{2\int_0^t {\rm e}^{-\frac{H(\psi^{-1}(s))}{\alpha}} ds }}
\leq
\int_0^{w_0}\frac{dw}{\sqrt{2\int_w^{w_0}{\rm e}^{\frac{2}{\alpha}[H(w)-H(\tau )]}d\tau }},
\end{eqnarray*}
with $0<w_0 = \psi^{-1}(t_0) <1= \psi^{-1} (0)$ since
$\psi $ is nonincreasing, and we choose $t_0>>1$ such that $h$ is
nonincreasing in $(0,w_0 ]$.

Since $h$ satisfies \rife{cong-1}, also $\widetilde{h}$ satisfies it, so that we  conclude the proof if we show that there exists a positive contant $c_0$ such that
\begin{equation} \label{canicula}
\widetilde{h}(w)\int_{w}^{w_0}{\rm e}^{\frac{2}{\alpha}[H(w)-H(\tau )]}\ d\tau \geq c_0>0 , \quad
\forall w\in (0,w_0).
\end{equation}
Indeed,  the only difficulty is near zero. To overcome it, we use that $h$ (hence $\widetilde{h}$) is nonincreasing in $(0,w_0 ]$, to obtain
\begin{eqnarray*}
\dys \widetilde{h}(w) \int_{w}^{w_0}{\rm e}^{\frac{2}{\alpha}[H(w)-H(\tau )]}\ d\tau
&\geq& \dys  \int_{w}^{w_0}\widetilde{h}(\tau ){\rm e}^{\frac{2}{\alpha}[H(w)-H(\tau )]}\ d\tau
\\
\dys &=&-\frac{\alpha {\rm e}^{\frac{2}{\alpha}H(w)}}{2}\int_{w}^{w_0}-\frac{2}{\alpha}\widetilde{h}(\tau ){\rm e}^{-\frac{2}{\alpha}H(\tau )}\
d\tau
\\
&=&-\frac{\alpha {\rm e}^{\frac{2}{\alpha}H(w)}}{2}\left[ {\rm e}^{ -\frac{2}{\alpha} H(\tau ) }\right]^{w_0}_{w} \dys=
-\frac{\alpha}{2}\frac{{\rm e}^{\frac{2}{\alpha} H(w)}}{{\rm e}^{\frac{2}{\alpha} H(w_0)}} +\frac{\alpha}{2}.
\end{eqnarray*}
Using the above inequality and the fact that
${\rm e}^{\frac{2}{\alpha}H(w)}$ is close to zero for $w$ small enough,  we
can choose $\overline{w}\in(0,w_0)$ such that
$$
\widetilde{h}(w)\int_{w}^{w_0}{\rm e}^{\frac{2}{\alpha}[H(w)-H(\tau )]}\ d\tau \geq
\frac{\alpha}{4},
$$
for $0<w< \overline{w}$. Thus the existence of $c_0$ such that \rife{canicula} holds is deduced.
\end{proof}

\begin{remark} \label{nol1}\rm
If $h$ is such that
$$
\lim_{s \to 0^{+}}\,\int_{s}^{1}\,h(t)\,dt = +\infty,
$$
there is no need to define the above function $\widetilde{h}$. Indeed, in this case, the proof of the above theorem works by using directly $h$ instead of $\widetilde{h}$.
\end{remark}

\begin{proof}[{\sl Proof of Theorem~\ref{teorema}}]
We are going to  prove that, up to a subsequence, the sequence $\{u_n\}$ of
finite energy solutions of   \rife{(P_n)}  converges to a  finite
energy solution of \rife{prob}.

By Case 1. of Lemma~\ref{lemaprop}, we obtain the existence of
constants $C_1,C_2>0$ such that
\begin{equation}\label{sarc}
\|u_n \|_{\huz} \leq C_1 \mbox{ and } \intO u_n g_n (x,u_n)\frac{|\D
u_n|^2}{1+\frac{1}{n} |\D u_n|^2}\leq C_2.
\end{equation}
Thus, up to a subsequence, we can assume that $u_n$ converges to
some $u\in H_0^1(\Omega )$
weakly
in $H_0^1(\Omega )$  and, by Rellich's Theorem, strongly in $L^2 (\Omega)$ and a.e. in $\Omega$.

Choosing
$
\frac{1}{\varepsilon}T_{\varepsilon} (u_n)
$
as test function   in \rife{(P_n)} and taking into account that  $f_n\leq f$ in $\Omega$, we deduce that
\begin{eqnarray*}
 \int_{\Omega} \frac{T_{\varepsilon} (u_n)}{\varepsilon}   g_n (x,u_n)\frac{|\D u_n|^2}{1+\frac{1}{n} |\D u_n|^2} \,
\leq                  \intO f_n \, \leq \intO f \,.
\end{eqnarray*}
If we take the limit as $\eps$ tends to zero, and we use that, by
Lemma~\ref{lemaprop}, $u_n>0$ in $\Omega$, we get
\begin{equation} \label{l1}
\int_{\Omega}   g_n (x,u_n) \frac{|\D u_n|^2}{1+\frac{1}{n} |\D u_n|^2}
\, =
\int_{\{u_n>0\}} \!\!\!  g_n (x,u_n) \frac{|\D u_n|^2}{1+\frac{1}{n} |\D
u_n|^2}  \, \leq \intO f \, .
\end{equation}
The proof will be concluded by proving the following   steps:

\noindent {\bf Step 1.} For every $ k>0$, $T_k (u_n)
\to T_k (u)$ strongly in $H^1_{\rm loc} (\Omega)$.

\noindent {\bf Step 2.} $u_n$ is
strongly convergent in $H^1_{\mbox{\tiny loc}} (\Omega)$.

\noindent {\bf Step 3.} We pass to the limit in \rife{(P_n)}.\\[2.0 ex]
\noindent
{\bf Step 1.} Here we want to prove  that
\begin{equation}\label{tronca}
\dys \lim_{n\to +\infty} \intO  |\nabla (T_k (u_n)-T_k (u)) |^2 \phi \, =0,\quad
\forall \phi \in C^{\infty}_c (\Omega) \mbox{ with }\,\phi\geq 0.
\end{equation}

Reasoning as in \cite{bgm}, we consider the function $\vp_{\la} (s)$  defined
in \rife{vpla} and we choose  $\vp_{\la}
(T_k (u_n)-T_k (u)) \phi $ as test function in \rife{(P_n)}: we have
$$
\intO  M(x,u_n) \D u_n \cdot \D (T_k (u_n)-T_k (u))  \vp_{\la} '(T_k (u_n)-T_k
(u)) \phi \, $$
$$ +\intO M(x,u_n) \D u_n \cdot \D \phi\,  \vp_{\la} (T_k
(u_n)-T_k (u)) \,
$$ $$
+\intO g_n(x,u_n)\frac{|\nabla u_n|^2}{1+\frac{1}{n}|\nabla u_n|^2} \vp_{\la}
(T_k (u_n)-T_k (u)) \phi \,  = \intO f_n\,\vp_{\la} (T_k (u_n)-T_k (u)) \phi\,
.
$$
Since  $T_k (u_n) \to T_k (u)$ weakly in $H^1_0
(\Omega)$ and  strongly in $\elle2$,   we note that
$$
\intO f_n\,\vp_{\la} (T_k (u_n)-T_k (u)) \phi\,
-\intO M(x,u_n)\D u_n \cdot \D \phi \, \vp_{\la} (T_k (u_n)-T_k (u))\,  = \eps (n).
$$

Moreover, choosing $ \omega_\phi\subset\subset \Omega$  with $\mbox{supp\,}
\phi \subset\omega_\phi$, we deduce, by Case (i) of Remark~\ref{recta} and by
the  nonnegativeness of both $g_n$ and $\vp_{\la}(k - T_k(u))$, that
$$
\intO g_n(x,u_n)\frac{|\nabla u_n|^2}{1+\frac{1}{n}|\nabla u_n|^2} \vp_{\la}
(T_k (u_n)-T_k (u)) \phi \,
$$
$$
\geq \int_{\{u_n\leq k\}} g_n(x,u_n)\frac{|\nabla u_n|^2}{1+\frac{1}{n}|\nabla
u_n|^2} \vp_{\la} (T_k (u_n)-T_k (u)) \phi \,
$$
$$
\geq - c_k(\omega_\phi)  \intO |\nabla T_k (u_n)|^2 |\vp_{\la} (T_k (u_n)-T_k (u)) |\phi \, .
$$
Thus
\begin{equation}                     \label{ccua}
\begin{array}{c}
\intO M(x,u_n) \D u_n \cdot \D (T_k (u_n)-T_k (u))  \vp_{\la} '(T_k
(u_n)-T_k (u)) \phi \, \\
-  c_k(\omega_\phi)  \intO  |\nabla T_k (u_n)|^2  |\vp_{\la} (T_k
(u_n)-T_k (u)) |\phi \,   \leq\eps (n) .
\end{array}
\end{equation}
Note that
$$
\intO M(x,u_n) \D  u_n\cdot \D (T_k (u_n)-T_k (u))  \vp_{\la} '(T_k (u_n)-T_k
(u)) \phi \chi_{\{u_n\geq k\}}  \,
$$
$$
= - \intO M(x,u_n) \D  u_n\cdot \D T_k (u)  \vp_{\la} '(k-T_k (u)) \phi
\chi_{\{u_n\geq k\}}  \, =\eps (n),
$$
so that, adding
$$
-\intO  M(x,u_n) \D T_k (u) \cdot \D (T_k (u_n)-T_k (u))  \vp_{\la} '(T_k
(u_n)-T_k (u)) \phi  \,  =\eps (n)
$$
in both sides of \rife{ccua} and since
$$
\intO  |\nabla T_k (u_n)|^2  |\vp_{\la} (T_k (u_n)-T_k (u)) |\phi  \,
$$
$$
\leq  2 \intO  |\nabla (T_k (u_n)-T_k (u))|^2  |\vp_{\la} (T_k (u_n)-T_k (u))
|\phi  \,
$$
$$
+  2\intO  |\nabla T_k (u)|^2  |\vp_{\la} (T_k (u_n)-T_k (u)) |\phi  \,
$$
$$
=  2 \intO  |\nabla (T_k (u_n)-T_k (u)) |^2  |\vp_{\la} (T_k (u_n)-T_k (u))
|\phi  \,  +\eps (n),
$$
we find, using also \rife{coe} (for the sake of brevity, we omit writing the argument $T_{k}(u_{n}) - T_{k}(u)$ for $\vp_{\la}$ and $\vp'_{\la}$),
$$
\intO   |\nabla (T_k (u_n)-T_k (u)) |^2 \Big[\alpha \vp_{\la}'
 -  2 c_k(\omega_\phi)   |\vp_{\la} |\Big] \phi  \,  \leq \eps (n).
$$
Choosing $\lambda  $ such that   \rife{vpla} holds  with $a=\alpha$ and $b=2 c_k(\omega_\phi)$,
we obtain \rife{tronca}.

\vspace{0.5cm}
\noindent {\bf Step 2.} We prove now  that the sequence $u_n$ is
strongly convergent in $H^1_{\mbox{\tiny loc}} (\Omega)$.

Let us choose $G_k (u_n)$ as test function in \rife{(P_n)} and
drop the  positive integral involving the lower order term.
By using \rife{coe}, and H\"older  and Sobolev  inequalities, we
have
$$
\intO |\D G_k (u_n)|^2 \,   \leq \frac{\mathcal{S} ^2}{\alpha^{ 2}}
\left(\int_{\{u_n\geq k\}} \!\!\! f^{\frac{2N}{N+2}}\
\right)^{1+\frac{2}{N}}\,,
$$
and the right hand side of the previous inequality is arbitrarily
small if $k$ is large enough. This and the convergence proved in
Step~1 of $T_k(u_n)$ in $H_0^1(\Omega )$  implies that $|\D u_n|^2$
is equiintegrable in every $\omega \subset \subset \Omega$.

    Moreover, since
$$
-{\rm div}(M(x,u_{n})\nabla u_{n}) = f_{n} - g_{n}(x,u_{n})\frac{|\nabla u_{n}|^{2}}{1 + \frac1n |\nabla u_{n}|^{2}}\,,
$$
and the right hand side is bounded in $\elle1$ by the assumptions on
$f$ and by \rife{l1}, we can apply Lemma~1 of \cite{BG} (see also
\cite{BM}) to deduce that, up to (not relabeled) subsequences,  $\D
u_n$ converges to $\D u$ a.e. in $\Omega$. Hence,  by Vitali theorem
$$
u_n \to u \qquad \mbox{in } H^1_{\rm loc} (\Omega).
$$

\noindent{\bf Step 3.}
Let us observe that, by  applying  Fatou
lemma in \rife{sarc} and \rife{l1}, we deduce 
that
$$
\intO ug (x,u)|\D u|^2\,  \leq C_2 \quad \mbox{and}\quad \int_{\Omega }   g
(x,u) |\D u |^2 \, \leq \intO f\,  \,,
$$
respectively. Therefore, to conclude the proof we only have to
prove that $u$ is a distributional solution of the problem (\ref{prob}). We begin
by passing to the limit on $n$ in the equation satisfied by $u_n$, i.e., in
$$
\int_{\Omega}M(x,u_n)\nabla u_n\cdot\nabla \phi \, + \int_{\Omega}
g_n(x,u_n)\frac{|\nabla u_n|^2}{1+\frac{1}{n}|\nabla u_n|^2}\phi \,
=
\int_{\Omega}f_n \phi \,   , \  \forall\phi\in C_c^\infty(\Omega) .
$$
First of all, the weak convergence of
$u_n$ to $u$ and the weak-$\ast$ convergence of $M(x,u_n)$  to $M(x,u)$ in $L^{\infty}(\Omega)$ implies that
\begin{equation}   \label{nete}
\lim_{n\rightarrow+\infty}\int_{\Omega}M(x,u_n)\nabla u_n\nabla \phi\
  =
   \int_{\Omega}M(x,u)\nabla u\nabla \phi\ , \quad  \forall \phi\in C_c^\infty(\Omega) .
\end{equation}
On the other hand,
if we fix  $\omega \subset\subset \Omega $, then, by Remark~\ref{recta},
$$
g_n(x,u_n(x)) \leq c_k (\omega),\ \  \forall n>>1, \mbox{ and  }
\forall x\in \omega \mbox{ satisfying } u_n(x)\leq k .
$$
Consequently, if $E\subset \subset \omega$ we have
$$
\dys
  \int_E |g_n(x,u_n(x))| \frac{|\nabla u_n (x)|^2}{1+\frac 1n
|\nabla u_n (x)|^2} \,
$$
$$\leq \dys \int_{E\cap \{u_n \leq k \}
}g_n(x,u_n) \frac{|\nabla u_n |^2}{1+\frac 1n |\nabla u_n |^2} \,  +
\int_{E\cap \{u_n \geq k \} } g_n(x,u_n) \frac{|\nabla u_n |^2}{1+\frac 1n |\nabla u_n|^2} \,
$$
\begin{equation}
                     \label{espi}
\leq \dys c_k (\omega) \int_{E\cap \{u_n \leq k \} }  |\nabla T_k (u_n )|^2 \,
+
\int_{ \{u_n \geq k \} } g_n(x,u_n)  \frac{|\nabla u_n  |^2}{1+\frac 1n |\nabla u_n
|^2} \, .
\end{equation}
Let $\varepsilon >0$ be fixed. Observe that if, for $ k>1$, we use $T_1
(G_{k-1} (u_n))$  as test function in \rife{(P_n)} and drop positive
terms, we deduce that
$$
\dys \int_{\{u_n \geq k\}} g_n(x,u_n)\frac{|\nabla u_n|^2}{1+\frac{1}{n}|\nabla
u_n|^2} \,  \leq \dys \int_{\{u_n \geq k-1\}} f_n \,  \leq \dys \int_{\{u_n
\geq k-1\}} f\,  .
$$
Thus, since the right hand side tends to $0$ uniformly in $n$ as $k$ diverges, we obtain the existence of $k_0 >1$ such that
$$
\dys \int_{\{u_n \geq k\}} g_n(x,u_n)\frac{|\nabla u_n|^2}{1+\frac{1}{n}|\nabla
u_n|^2} \,
  \leq \frac{\varepsilon}{2}
  , \quad \forall k\geq k_0, \ \forall n \in \mathbb{N}.
$$
Moreover, since $T_k (u_n)$ is
strongly compact in $H^1_{\rm loc} (\Omega)$, there exist $n_{\eps}$, $\delta_{\eps}$ such that  for every $E\subset \subset \Omega$ with
$
\mbox{meas}\, (E)<\delta_{\eps}
$ we have
$$
\int_{E\cap \{u_n \leq k \} }  |\nabla T_k (u_n )|^2 \, <
 \frac{\eps}{2 c_k (\omega)  } , \quad  \forall n\geq n_\eps .
 $$
In conclusion, by \rife{espi}, taking $k\geq k_0$ we see that $\mbox{meas\,}(E) <\delta_\varepsilon$ implies
$$
\dys
  \int_E |g_n(x,u_n(x))| \frac{|\nabla u_n (x)|^2}{1+\frac 1n
|\nabla u_n (x)|^2} \,
          \leq  \eps , \quad \forall n\geq n_\eps ,
$$
 i.e., the sequence
$g_n(x,u_n) \frac{|\nabla u_n
|^2}{1+\frac 1n |\nabla u_n |^2}$ is equiintegrable. This, together with its
a.e.\ convergence to $g (x,u )  |\nabla u  |^2$, implies by Vitali theorem that
$$
\lim_{n\to +\infty} \int_{\Omega } g_n(x,u_n)\frac{|\nabla
u_n|^2}{1+\frac{1}{n}|\nabla u_n|^2} \phi \,  =\intO   g  (x,u) |\D u |^2 \phi ,
\quad
\forall \phi\in C^{\infty}_c (\Omega)
.
$$

Therefore, using the above limit,  \rife{nete} and since $f_n$ tends to $f$ strongly in $L^{1}(\Omega)$ we conclude that
$$
\int_{\Omega}M(x,u)\nabla u \nabla \phi \, + \int_{\Omega} g(x,u)|\nabla
u|^2\phi\,
=
\int_{\Omega}f \phi \,,  \quad \forall\phi\in C_c^\infty(\Omega).
$$
\end{proof}

\begin{remark}              \label{continuidad} {\rm
In addition, if $f\in L^q(\Omega )$ with $q>N/2$, then the solution $u$ given
by Theorem~\ref{teorema} is continuous in ${\Omega}$. Indeed,   by using
$\psi=T_m(G_k(u))$, with $m>k$,  as test function in \eqref{enerfin}, it is easy to adapt the idea of Stampacchia (\cite{S}) in order to obtain
\salta{
$$
\alpha\intO |\D G_k (u)|^2 \, \leq \intO f \,G_k (u) ,
$$}
that $u\in L^\infty (\Omega)$. Now, consider a function $\zeta\in
C^\infty (\Omega) $ with $0\leq \zeta(x)\leq 1$, for every
$x\in\Omega$ and compact support in a ball $B_\rho$ of radius
$\rho>0$, and set $A_{k,\rho}=\{x\in B_\rho\cap \Omega: u(x)>k\}$.
Following the idea of the proof of Theorem~1.1 of Chapter~4 in
\cite{lad}, take $\phi =\gku\zeta^2$ as test function in
\rife{enerfin} to deduce by \rife{coe} and H\"{o}lder's inequality that
$$
\alpha \int_{A_{k,\rho}} |\D u|^2 \zeta^2
\leq
\|f\|_{L^q(\Omega)}\|u\|_{L^\infty(\Omega)} (\mis{A_{k,\rho}})^{1-\frac 1q} +
2\beta \int_{A_{k,\rho}} |\D u| |\D \zeta| \zeta G_k(u).
$$
Using again Young's inequality we get
$$
\int_{A_{k,\rho}} |\D u|^2 \zeta^2
                          \leq
 \frac{2\|f\|_{L^q(\Omega)}\|u\|_{L^\infty(\Omega)}}{\alpha}
(\mis{A_{k,\rho}})^{1-\frac 1q}
       + \frac{4\beta}{\alpha^2} \int_{A_{k,\rho}} |\D \zeta|^2 G_k^2(u).
$$
In particular,  if
for $\sigma \in (0,1)$ we choose $\zeta$ such that it is constantly equal to~1 in
the concentric ball $B_{\rho- \sigma
\rho}$ (to $B_\rho$) of radius $\rho- \sigma\rho$ and
$|\D \zeta|< \frac{1}{\sigma\rho}$, we obtain
\begin{eqnarray*}
\dys \int_{A_{k,\rho-\sigma \rho}} |\D u|^2 \leq   \gamma
\left(1+\frac{1}{\sigma^2\rho^{2(1-\frac{N}{2q})}}
\max_{A_{k,\rho}}(u-k)^2\right) (\mis{A_{k,\rho}})^{1-\frac 1q}
,
\end{eqnarray*}
where $\gamma=\max \left\{
\frac{2\|f\|_{L^q(\Omega)}\|u\|_{L^\infty(\Omega)}}{\alpha},
\frac{4\beta}{\alpha^2}\omega_N^{\frac1q}\right\}$ with  $\omega_N$ denoting
the measure of the unit ball of $\rn$.

This means that for $\delta >0$ small enough and  every $M\geq \|u\|_{\lio} $,
the function $u$ belongs to the class $\mathcal B_2(\Omega, M, \gamma, \delta,
\frac1{2q})$ with $2q>N $ {(see \cite{lad}, pag. 81). Applying Theorem~6.1 of
\cite{lad} we deduce that $u$ is H\"{o}lder continuous in $\Omega$.
}}
\end{remark}

\setcounter{equation}{0}
\section{Further Existence Results}\label{fur}

\subsection{Existence for data in $\elle1$}\label{sec3}

In this section we prove Theorem~\ref{th1}. In this
case, taking advantage of Theorem~\ref{teorema}, we approximate problem
\rife{prob} by
\begin{equation}    \label{pn}
      \begin{cases}
     \displaystyle  -\mbox{div\,}(M(x,u_n)\nabla u_n) +g (x,u_n ) |\nabla u_n|^2
       =f_n\quad & \mbox{in } \Omega, \\
\hfill u_n=0 \hfill & \mbox{on } \partial \Omega,
        \end{cases}
\end{equation}
where $f_n=T_n (f)$.

Note that the existence of a nonnegative finite energy solution $ u_n\in H^1_0 (\Omega)\cap
C( {\Omega })$ such that $g(x,u_n ) |\nabla u_n|^2\in L^1
(\Omega)$ follows from Theorem \ref{teorema} and Remark~\ref{continuidad}.

\begin{lemma}\label{lemma}\sl
If $f\in L^1(\Omega )$ satisfies \rife{cona}, $g(x,s)$ satisfies
\rife{cong-2} (with $h(s)$ satisfying \rife{cong-1}),  and $u_n$ is a solution
of \rife{pn}, then
\begin{enumerate}

\item[(i)]  $u_n$ is bounded in $\mathcal M^{\frac{N}{N-2}}(\Omega)$ and $|\D u_n| $ is bounded in $\mathcal M^{\frac{N}{N-1}} (\Omega)$;

\item[(ii)]   up to  subsequences, the sequence $u_n$ is weakly convergent to some
$u$ in $W_0^{1,q} (\Omega )$ for every $q\in [1,\frac{N}{N-1})$;

\item[(iii)]  for any
$k>0$ and for any $\omega\subset \subset \Omega$,
$$
T_k (u_n)\to T_k (u)\ \ \ \text{in }\ \ H^1(\omega).
$$
\end{enumerate}
\end{lemma}

\begin{proof}[{\sl Proof}] (i) Taking
$T_{k} (u_n) $ as test function in \rife{pn} and using \rife{coe}, we have
$$
\alpha\intO |\D T_{k} (u_n)|^2 \, + \intO g (x,u_n ) T_{k} (u_n) |\D u_n|^2 \,
\leq k \|f_n\|_{\elle1} .
$$
Since $0\leq f_n \leq f$ and  $g(x,u_n) \geq 0$, we have
\begin{equation}
                      \label{circulo}
\alpha\intO |\D T_{k} (u_n)|^2  \ \leq  k  \|f\|_{\elle1}.
\end{equation}
Standard estimates (see \cite[Lemmas~4.1 and 4.2]{b6}) imply that $u_n$ is
bounded in $\mathcal M^{\frac{N}{N-2}}(\Omega)$ and that $|\D u_n| $ is bounded
in $\mathcal M^{\frac{N}{N-1}} (\Omega)$.

\noindent (ii) Let $1\leq q<\frac{N}{N-1}$. By the preceding
case and by
the embedding \rife{imm}, we deduce that
$u_n$ is bounded in $W^{1,q}_0 (\Omega)$  and thus, passing to a subsequence if
necessary, there exists $u $ such that $u_n \rightharpoonup  u$ weakly in
$W^{1,q}_0 (\Omega)$.

\noindent (iii) Our aim is to show that
$$
\dys  \lim_{n\to +\infty} \intO |\D (T_k (u_n) -T_k (u))|^2\phi =0\, ,
\qquad \forall \phi \in C_c^\infty (\Omega )\, ,\, \,  \phi\geq 0.
$$

Here we adapt to our case a technique to obtain the strong convergence of
truncations first introduced in \cite{leoporr} (see also \cite{Porretta}). Let
us choose $\vp_{\la } (w_n ) \phi $ as test function in \rife{pn} where
$\vp_{\la } (s)$ has been defined in \rife{vpla} and
$$
w_n =T_{2k} [u_n - T_l (u_n) + T_k (u_n) - T_k (u)],\qquad 0<k<l.
$$
Thus we have
\begin{equation}\label{1}
\begin{array}{l}
\dys \intO M(x,u_n)\D u_n \cdot \D w_n \vp_{\la}' (w_n)\phi  \,   +
\dys \intO M(x,u_n)\D u_n \cdot \D \phi   \vp_{\la} (w_n) \, \\\\
\dys\quad +\intO  g (x,u_n) |\D u_n|^2 \vp_{\la}(w_n) \phi\,
= \intO f_n \,\phi  \vp_{\la} (w_n)\,  .
\end{array}
\end{equation}

\noindent Observing that $\D T_k (u_n)=0$ if $u_n>k$ and $\D w_n\equiv 0 $ if $u_n \geq {  2k +l} \equiv
\mathcal{K}$
(we recall that $l>k$), we have
$$
\begin{array}{l}
\dys \intO M(x,u_n) \D u_n \cdot \D w_n \vp_{\la}' (w_n)\phi \,  \\\\
\dys  \quad= \intO
M(x,u_n)\D T_k (u_n) \cdot \D (T_k (u_n) -T_k (u)) \vp_{\la}' (w_n)\phi \, \\\\
\dys\qquad + \int_{\{u_n\geq  k\}} M(x,u_n)\D T_\mathcal{K} (u_n) \cdot \D T_{2k}(G_l (u_n ) + k
-T_k (u)) \vp_{\la}' (w_n)\phi \,  .
\end{array}
$$
Moreover, using that
\begin{eqnarray*}
\D T_{\mathcal{K}} (u_n) \cdot \D (G_l (u_n )  -T_k (u))&=& \D T_{\mathcal{K}} (u_n) \cdot \D G_l
(u_n )  - \D T_\mathcal{K} (u_n)  \D T_k (u) \\
&\geq& - \D T_{\mathcal{K}} (u_n) \cdot \D T_k (u)  ,
\end{eqnarray*}
  we have
$$
\begin{array}{l}
\dys  \int_{\{u_n> k \}\cap \{ \, G_l (u_n )  -T_k (u) \leq k\}}
M(x,u_n)\D T_{\mathcal{K}} (u_n) \cdot \D (G_l (u_n )  -T_k (u)) \vp_{\la}'
(w_n)\phi \,
\\
\quad  \dys\geq-  \int_{ \{G_l (u_n ) + k -T_k (u) \leq 2k\}}| M(x,u_n)\D
T_{\mathcal{K}} (u_n) \cdot \D  T_k (u)| \vp_{\la}' (w_n)\phi
\chi_{\{u_n>k\}}\, ,
\end{array}
$$
and thus, since the above integral tends to zero as $n$ diverges,
\begin{equation}                  \label{quela}
                  \begin{array}{l}
\dys \intO M(x,u_n) \D u_n \cdot \D w_n \vp_{\la}' (w_n)\phi \,  \\\\
\dys\qquad\geq  \intO
M(x,u_n) \D T_k (u_n) \cdot \D (T_k (u_n) -T_k (u)) \vp_{\la}'
(w_n)\phi \,  +\eps (n).
\end{array}
\end{equation}

On the other hand, since $G_l (u_n) +k-T_k (u) \geq 0$,
$$
\begin{array}{l}
\dys\intO  g (x,u_n)  |\D u_n|^2 \vp_{\la}(w_n) \phi \,
                                                        \dys\geq \int_{\{u_n\leq k\}}  g (x,u_n) |\D u_n|^2 \vp_{\la}(w_n) \phi \,  \,.
\end{array}
$$

Thanks to Case (ii) of Remark~\ref{recta}  applied to a subset $\omega_\phi
\subset \subset \Omega$ with $\mbox{ supp\,} \phi\subset \omega_\phi$, we have
$g(x,u_n(x)) \leq c_{k}(\omega_\phi) $ for every $x\in \omega$ with $
{u_n(x)}\leq k$. Then, we get
$$
\begin{array}{l}
\dys \left| \int_{\{u_n\leq k\}}  g (x,u_n)  |\D u_n |^2 \vp_{\la}(T_k
(u_n) -T_k (u)) \phi \,   \right|
\\\\
\dys \leq c_{k}(\omega_\phi) \intO |\D T_k (u_n)|^2  |\vp_{\la}(T_k (u_n) -T_k
(u))| \phi \,
\\\\
\dys  \leq 2 c_{k}(\omega_\phi) \intO |\D (T_k (u_n)-T_k (u))|^2
|\vp_{\la}(T_k (u_n) -T_k (u))| \phi  \,
\\\\
\qquad+2 c_{k}(\omega_\phi) \intO |\D T_k (u)|^2  |\vp_{\la}(T_k (u_n)
-T_k (u))| \phi\,.
\end{array}
$$

\noindent Note that the last integral tends to $0$ as $n$ diverges since
$\vp_{\la}(T_k (u_n) -T_k (u))$ converges to zero in the weak-$\ast$ topology
of $\lio$ and $T_k (u)\in \huz$.  Therefore, we deduce
from this, \rife{1} and \rife{quela} that
$$
\begin{array}{l}
\dys \intO M(x,u_n)\D T_k (u_n)\cdot  \D (T_k (u_n) -T_k (u)) \vp_{\la}' (w_n)\phi
\,  \\\\
\dys-2c_{k}(\omega_\phi) \intO |\D(T_k(u_n)-T_k (u))|^2
|\vp_{\la}(w_n)|\phi \,  \\\\
\dys
\leq \intO f_n \phi  \vp_{\la} (w_n)\ - \intO M(x,u_n) \D u_n \cdot \D \phi
\vp_{\la} (w_n)\,   +\eps(n),
\end{array}
$$
and adding to both sides of the previous inequality
$$
- \intO M(x,u_n)\D T_k (u)  \cdot \D (T_k (u_n) -T_k (u)) \vp_{\la}'
(w_n)\phi\, =\eps(n),
$$
we find from \rife{coe},
$$
\begin{array}{l}
 \displaystyle \intO |\D   ( T_k (u_n)- T_k (u)  )|^2  \Big[ \alpha
\vp_{\la}' (w_n)-2{c_{k}(\omega_\phi)} |\vp_{\la}(w_n)|
\Big] \phi \,
\\
\displaystyle \leq \intO f_n \phi  \vp_{\la} (w_n) \ - \intO M(x,u_n)\D u_n
\cdot \D \phi\, \vp_{\la} (w_n) \,  +\eps(n).
\end{array}
$$
Choosing $\lambda$ such that $\vp_\la$ satisfies \rife{vpla}
with $a=\alpha$ and $b=2c_k(\omega_\phi)$, we get
$$
\begin{array}{l}
\dys\frac\alpha2\intO   |\D (T_k (u_n) -T_k (u))|^2\phi  \,   \\\\
\dys \leq \intO f_n \phi  \vp_{\la} (w_n)\  - \intO M(x,u_n)\D u_n \cdot \D
\phi \, \vp_{\la} (w_n) \,  +\eps(n).
\end{array}
$$
Moreover, $w_n$ a.e. (and weakly-$\ast$ in  $\lio$) converges towards $w=T_{2k}
(G_l (u))$ and thus, recalling that $\D u_n \to \D u $ weakly in
$(L^q(\Omega))^N$, $q<N/(N-1)$,

$$
\begin{array}{l}
\dys \lim_{n\to +\infty} \intO f_n \phi  \vp_{\la} (w_n) \,  - \intO M(x,u_n)\D
u_n \cdot \D \phi   \vp_{\la} (w_n) \,
    \\\\ \qquad   =
       \intO f  \phi  \vp_{\la} (w) \,  -  \intO M(x,u) \D u  \cdot \D \phi   \vp_{\la} (w)\, .
\end{array}
$$

\noindent Consequently, using \rife{coe}
$$
\frac\alpha2\intO   |\D (T_k (u_n) -T_k (u))|^2\phi \,
\leq \intO f  \phi  \vp_{\la} (w)  \, \dys -  \intO M(x,u) \D u  \cdot \D \phi\,
\vp_{\la} (w)  \,  +\eps(n)
$$ $$
\dys  \leq \vp_{\la} (2k) \int_{\{u\geq l\} } (f  +  \beta |\D u|| \D \phi  |)
\,  +\eps(n).
$$
\noindent
       Since the last integral tends to zero  as $l$ diverges,
(iii) is proved.
\end{proof}

Now, we prove our main result concerning $\elle1$ data: \noindent
\begin{proof}[{\sl Proof of Theorem \ref{th1}}]
We begin by proving the first part of the theorem, i.e.\ that there exists a
solution $u\in W^{1,q}_0 (\Omega)$, for every $q<\frac{N}{N-1}$, of problem
\rife{prob}.
We first observe that we deduce from the results of \cite{BM} that $\D u_n \to \D u$  a.e., and from Lemma \ref{lemma} the estimates on $u_n$ and $|\D u_n|$ in $\mathcal M^{\frac{N}{N-2}}(\Omega)$ and $\mathcal M^{\frac{N}{N-1}}(\Omega)$
respectively. Thus
$u_n \to  u$ strongly in $W^{1,q}_0 (\Omega)$, for every
$q<\frac{N}{N-1}$. Arguing as in the proof of Theorem~\ref{teorema}, we can show that, choosing $\frac{1}{\varepsilon}T_\varepsilon (u_n)$ as test function in \rife{pn} and applying Fatou lemma, we have $ g (x,u )  |\D u |^2\in L^1 (\Omega)$.

In order   to prove that for all   $\omega
\subset \subset \Omega$, $\{g (x,u_n) |\D u_n|^2\}$ is strongly {convergent} in
$L^1 (\omega)$ to $g (x,u) |\D u|^2$, it suffices to show  the local uniform
equiintegrability of such sequence. To prove the claim,
we choose $T_1 (G_{k-1} (u_n))$  (for $k>1$) as test
function in the equation \rife{pn} and we deduce, by dropping the first
positive term (in virtue of \rife{coe}), and since $f_n\leq f$, that
\begin{equation}\label{equi}
\dys \int_{\{u_n \geq k\}} g (x,u_n)  |\D u_n|^2 \,  \leq \int_{\{u_n \geq k-1\}} f\,  \, .
\end{equation}
By a similar argument to the one used in Step~3 of the proof of
Theorem~\ref{teorema}, we prove the claim.  Indeed, let
$E\subset \omega\subset\subset\Omega$ be a measurable set. By
Remark~\ref{recta}-(ii) and \rife{equi}, we have, $\forall k
1$,
$$
\begin{array}{l}
\dys \int_E g (x,u_n)  |\D u_n|^2\   =
\int_{E\cap \{u_n \leq k \}} g (x,u_n)  |\D u_n |^2 \,
\\\\ +
 \dys \int_{E\cap \{u_n \geq k \}} g (x,u_n)  |\D u_n|^2 \,
\leq
 c_{k}(\omega) \dys \int_{E\cap \{u_n \leq k \}}   |\D T_k (u_n) |^2 \, \\\\+
\dys \int_{ \{u_n \geq k \}} g (x,u_n) |\D u_n|^2 \,
\leq
 c_{k}(\omega) \dys \int_{E}   |\D T_k (u_n) |^2 \,  +
\dys \int_{ \{u_n \geq k-1 \}} f \, .
\end{array}
$$
Since $\mis(\{x\in\Omega \, :\, u_n \geq k-1\})$ tends to zero (uniformly with
respect to $n$) as $k$ tends to ${+\infty}$ (because of the boundedness of $\{
u_n\}$ in the space $\mathcal M^{N/(N-2)} (\Omega )$ by
Lemma~\ref{lemma}-(ii)), we obtain that the last integral in the above
inequalities tends to zero as $k$ goes to $+\infty$. This, and the local
equiintegrability of $|\D T_k(u_n)|^2$  (by Lemma~\ref{lemma}-(iii)), then show
the local equiintegrability of $\{ g (x,u_n)  |\D u_n|^2\} $.

Using moreover that $\D u_n \to \D u$ a.e., we conclude by Vitali theorem that
\begin{equation}\label{cpt1}
g (x,u_n)  |\D u_n|^2\to g (x,u)  |\D u|^2 \quad \mbox{in }\,L^1 (\omega)\,, \,\,\forall \omega \subset \subset \Omega\,.
\end{equation}
Now, using \rife{cpt1} and the strong convergence of $\D u_n $ to $\D u$ in
$(\elle{q})^N$, for every $ q<\frac{N}{N-1}$,
we can pass to the limit in \rife{pn}
to show that $u$ is  a solution for \rife{prob}.

\smallskip

In order to prove the second part of the theorem, we simply note
that we can fix $k\geq \max\{s_0,1\}$ so that
\rife{coerc} and \rife{equi} imply
\begin{equation}\label{geca}
\mu \int_{\Omega } |\D
G_{k} (u_n)|^2 \,  = \mu \int_{\{ u_n \geq k\}} |\D u_n|^2  \,  \leq
\int_{\{u_n\geq k-1\}} f \,  \leq \|f\|_{\elle1}.
\end{equation}
Hence, taking into account both \rife{circulo} and \rife{geca}, we have
$$
\into |\nabla u_n|^2 \, = \into |\nabla T_{k} (u_n)|^2 \,  + \into |\nabla
G_{k} (u_n)|^2 \,  \leq
\left(\frac{k}{\alpha}+\frac{1}{\mu}\right)\|f\|_{\elle1},
$$
i.e., the boundedness of the sequence $\{u_n\}$ in $\huz$. This implies that
the solution $u$, which is the limit of (a subsequence of) $\{u_n\}$, belongs
to $\huz$.
\end{proof}

\begin{remark}\rm
Actually, if \rife{coerc} holds, it is possible to prove, in this latter case,  that the approximate sequence $u_n$ is strongly convergent to $u$ in $H^1(\omega)$, for every $\omega\subset\subset\Omega$. Indeed, due to the a.e. convergence of $\nabla u_n$ to $\nabla u$ in $\Omega$, it suffices to check the equiintegrability of $|\nabla u_n|^2$ in every $\omega\subset\subset\Omega$.
To do that, we take a measurable set $E\subset \omega\subset\subset\Omega$, and we observe that, thanks to \rife{geca}, for any $k\geq \max\{s_0,1\}$, we can write
\begin{eqnarray}
\dys \int_{E}|\nabla u_n|^2 \,  &=&\int_{E}|\nabla T_k (u_n)|^2 \,
+\int_{E}|\nabla G_k (u_n)|^2 \,
  \nonumber
\\
&\leq& \int_{E}|\nabla T_k (u_n)|^2 \, + \frac{1}{\mu}\int_{\{u_n \geq k-1\}} f
\, . \label{unifor}
\end{eqnarray}
Therefore, using again both the boundedness of $u_n$ in $\mathcal{M}^{\frac{N}{N-2}}(\Omega)$ and  the equiintegrability of $|\nabla T_k(u_n)|^2$  in $\omega$ given by Lemma \ref{lemma}, we see  that  \rife{unifor} yields the desired result.
\end{remark}

\subsection{Non-divergence operators}\label{nodi}

In this section we sketch the proof of Theorem \ref{notdiv} without
giving all the details since they are straightforward adaptations of
the applied arguments  in the proof of Theorem \ref{teorema}.

\begin{proof}[{\sl Proof of Theorem~\ref{notdiv}}]
We denote by $P(x)$ the vector field whose  $i^{\mbox{\tiny th}}$
component is \\ $P_i(x)= b_i(x) +\sum_{j=1}^N \frac{\partial
a_{ij}}{\partial x_j}(x)$ ($i=1,\ldots,N$) and by $M(x)$ the
transpose  of the matrix $(a_{ij}(x))_{i,j=1,\ldots,N}$. Let  $g_n
(x,s)$ also  be given by \rife{condult}. Consider the sequence
$u_n\in \huz\cap\lio$ of solutions for the problem
\be\label{nodivapp}
\begin{cases}
\displaystyle
-\div (M(x)\D u_n ) + P (x) \cdot \D u_n 
+ g_n(x, u_n) |\nabla u_n|^2=
f & \mbox{in } \Omega, \\
\hfill u_n =0 \hfill &\mbox{on } \partial \Omega.
\end{cases}
\ee

The proof is divided into several steps.

{\bf Step 1.}  $\lio$ estimate. Using the ideas of \cite{BMP}, we
choose $v=e^{2\lambda G_k (u_n)} -1 $, with $\la >>1$ as test
function in the weak formulation of \rife{nodivapp} to prove  that
the sequence $\{u_n\}$ is bounded in $\lio$.

{\bf Step 2.} $\huz$ estimate.  By the previous $\lio $ estimate, it
is easy to see that $\{u_n\}$ is bounded in  $\huz$, and so $u_n $
weakly converges in $\huz$ to a function $u$ in $\huz\cap \lio$.
Moreover, arguing as in Remark \ref{continuidad}, it is clear that
both $u_n$ (for every $n\in\na$) and $u$ are continuous in $\Omega$.

{\bf Step 3.} Estimate on the lower order term. Choosing
$\frac{T_\varepsilon(u_n)}{\varepsilon}$ as test function in
\eqref{nodivapp} and taking limit as $\varepsilon$ tends to zero, we
deduce, for some $C_1>0$, that
$$ \intO \dys g_n(x, u_n) |\nabla u_n|^2 \leq \intO|f| + C_1. \,$$

{\bf Step 4.} Uniform bound from below for $u_n$ in compact sets.
Observe that $u_n$ are supersolutions of the equation \be\label{cha}
-\div (M(x)\D u ) + P (x) \cdot \D u  +h(u) |\nabla u|^2 = f \quad
  \mbox{in } \Omega .
\ee
  If, for $H(s)$  defined in \rife{acca} and $\phi\in C^{\infty}_c (\Omega)$,  we take
  $e^{-\frac{H (u)}{\alpha}} \phi$ as test function in \eqref{cha},
  we see that $v_n= \psi (u_n) $ are subsolutions of
  \be\label{nodivch}
\displaystyle -\div (M (x) \D v ) + P(x) \cdot \D v + b(v)f (x)
=0 \quad
  \mbox{in } \Omega,
 \ee
where $b(s)$ has been defined in \rife{b(s)} and we recall that it
satisfies the Keller-Osserman condition (see \rife{k-o}). Hence, by
Theorem~\ref{thleoni} in the Appendix, we conclude that for every
$\omega \subset \subset \Omega$, there exists $C_{\omega}$ such that
$v_n=\psi(u_n) \leq C_{\omega}$ in $\omega$. Therefore, $u_n \geq
c_{\omega}>0$ in $\omega$, with $c_\omega = \psi ^{-1} (C_\omega)$.

{\bf Step 5.} Compactness of $\{u_n\}$ in $H_{\mbox{\tiny loc}}^1
(\Omega)$. For $\vpla (s)$ defined in \rife{vpla} and $\phi\in
C_c^{\infty} (\Omega)$, we  choose $\vpla (u_n-u)\phi $ as test
function in the weak formulation of \rife{nodiv} and  we note that
the ideas of Theorem~\ref{teorema} works since the \emph{new term}
that appears in the equation does not lead to any further difficulty
because it is linear with respect to $\nabla u$. Thus we conclude
that, up to a subsequence,
$$
u_n \to u \quad \mbox{in} \,\,H^1_{\rm loc} (\Omega)\,.
$$

{\bf Step 6.} Passing to the limit. By Step 5, we pass to the limit
in the weak formulation of \rife{nodivapp} to deduce  that  $u$ is a
solution  for
\begin{equation}   \label{nodiv}
\begin{cases}
-\div (M(x)\D u ) + P (x) \cdot \D u  + g(x, u) |\nabla u|^2 =
f & \mbox{in } \Omega, \\
\hfill u=0 \hfill &\mbox{on } \partial \Omega .\,
\end{cases}
\end{equation}
Since the coefficients $a_{ij}$ are Lipschitz continuous on
$\Omega$, we see that $u$ solves \eqref{nodive}. Finally, by Step 3
we conclude that $g(x,u)|\D u|^2\in L^1(\Omega)$.
\end{proof}

\setcounter{equation}{0}
\section{Nonexistence results}

This section is devoted to study  nonexistence of solutions for
\rife{prob}. We begin by observing that if the function
$g(x,s)$ satisfies condition \rife{hcero} with $h$ such that
\rife{contador} and \rife{pereiro} hold, then we can change $h$ by a
smaller function $\overline{h}$ which, in addition to
\rife{contador} and \rife{pereiro}, also satisfies $\overline{h}(s)=0$
for every $s>1$. Indeed, if $s_0$ is the point where $h$ attains its
minimum value in $[\frac{1}{2},1]$, then it suffices to define
$$
\overline{h}(s)=\left\{ \begin{array}{cl} (h(s)- h(s_0))^+ &
\mbox{if }s\in (0,s_0],
\\
0 & \mbox{if }s>s_0 .
\end{array}
\right.
$$
Consequently, without loss of generality, we will assume in the following that
condition \rife{hcero} holds with $h$ satisfying \rife{contador},
\rife{pereiro}, and
\begin{equation}
                  \label{h10bis}
h(s)=0,\ \ \
\forall  s\geq 1 .
\end{equation}

Let us consider the function $G:(0,+\infty)\to
(0,+\infty)$ given by
$$
G(s) = \displaystyle {\rm e}^{\displaystyle \int_1^s \frac{h(t)}{\beta}dt}  \quad\mbox{for every } s>0,
$$
where $\beta$ is given by \rife{coe}.
Observe that, by \rife{contador}, the function $G$
can be continuously extended to $[0,+\infty)$  setting $G(0)=0$. 
    Moreover, we also define the function
$\sigma:[0,+\infty)\to[0,+\infty) $ by setting $\sigma(0)=0$ and
$$
\sigma(s) = {\rm e}^{\displaystyle \int_1^s \sqrt{h(t)}dt} \quad  \mbox{for every } s>0.
$$

Observe that, thanks to \rife{contador} and \rife{pereiro}, we have that
$\sigma\in C^1([0,+\infty))$, $\sigma'(0) = h_0$ and $\sigma(s)=0$ if and only
if $s=0$. As a consequence of \rife{h10bis}, $\sigma(s)=1$ for every $s>1$ and
$\sigma(s)\leq 1$ for every $s\geq 0$. The next lemma is the key for the proof of Theorem \ref{teoremah}.

\begin{lemma}   \label{lema1h}\sl
Assume \rife{contador} and \rife{pereiro}. Then the function
\begin{equation}    \label{fih}
\varphi(s)= \left\{
\begin{array}{cl}
\displaystyle\frac{\displaystyle \int_0^s G(t)[\sigma'(t)]^2 dt}{G(s)} &
\ \ \mbox{if } s>0,
\\\\
0& \ \ \mbox{if } s=0,
\end{array}
\right.
\end{equation}
is a continuously differentiable function on $[0,+\infty)$ that satisfies the
ordinary differential equation
\begin{equation}\label{rh}
\left\{
\begin{array}{cl}
\displaystyle\varphi'(s)+\frac{h(s)}{\beta}\varphi(s) = [\sigma'(s)]^2, &
\mbox{on } [0,+\infty),
\\
\varphi(0)=0.
\end{array}
\right.
\end{equation}
Moreover, the following inequality holds:
\begin{equation} \label{estima}
\varphi(s)\leq \beta[\sigma(s)]^2, \quad \forall s>0.
\end{equation}
\end{lemma}

\begin{proof}[{\sl Proof.}]
The first part of the proof is straightforward except for checking that
$\varphi$ is differentiable at zero and $\varphi'$ is continuous at zero. In
order to do it, we note firstly that $\varphi$ is continuous at zero. Indeed,
since $G$ is nondecreasing and $[\sigma']^2$ is continuous in $[0,+\infty)$ we
have
\begin{eqnarray*}
0\leq \lim_{s\to 0^+} \varphi (s) &=& \lim_{s\to 0^+} \displaystyle
\frac{\displaystyle\int_0^s G(t)[\sigma'(t)]^2 dt}{G(s)}
\leq\lim_{s\to 0^+} \int_0^s [\sigma'(t)]^2 dt =0.
\end{eqnarray*}

\noindent Now we observe that, using the L'H\^{o}pital Rule, \rife{contador} and
\rife{pereiro},
\begin{eqnarray*}
\varphi'(0)
&=& h_0^2 - \lim_{s\to 0^+} \frac{h(s)\displaystyle\int_0^s
G(t)\left[\sigma'(t)\right]^2dt}{\beta G(s)}
\\
&=& h_0^2 - h_0^2\lim_{s\to 0^+} \frac{\displaystyle G(s)[\sigma'(s)]^2}{2\beta
\sigma(s)\sigma'(s)G(s) + h(s)[\sigma(s)]^2 G(s)}
\\
&=& h_0^2 - h_0^2 \lim_{s\to 0^+} \frac{\displaystyle 1}{2\beta
\frac{1}{\sqrt{h(s)}} + 1}
= h_0^2 - h_0^2 = 0.
\end{eqnarray*}
\noindent Hence $\varphi$ is differentiable at zero and $\varphi'$ is
continuous at zero.

\noindent 
In order to prove inequality \rife{estima}, we observe that since
$[\sigma'(s)]^{2} = [\sigma(s)]^{2}\,h(s)$, then
$$
\varphi(s) =
\frac{\beta}{G(s)}\,\int_{0}^{s}\,G(t)\,\frac{h(t)}{\beta}\,[\sigma(t)]^{2}\,dt.
$$
Since
$$
G(t)\,\frac{h(t)}{\beta} = \frac{d}{dt}\,G(t),
$$
we can integrate by parts to find (recall that $G(0) = \sigma(0) = 0$)
$$
\begin{array}{r@{\hspace{2pt}}c@{\hspace{2pt}}l}
\dys
\varphi(s)
& = &
\dys
\frac{\beta}{G(s)}\,\left[G(t)\,[\sigma(t)]^{2}\right]_{t = 0}^{t = s}
-
\frac{2\beta}{G(s)}\,\int_{0}^{s}\,G(t)\,\sigma(t)\,\sigma'(t)\,dt
\\
& = &
\dys
\beta\,[\sigma(s)]^{2}
-
\frac{2\beta}{G(s)}\,\int_{0}^{s}\,G(t)\,[\sigma(t)]^{2}\,\sqrt{h(t)}\,dt
\\
& \leq &
\beta\,[\sigma(s)]^{2},
\end{array}
$$
since all the functions in the last integral are nonnegative.
\end{proof}

\noindent
\begin{proof}[{\sl Proof of Theorem~\ref{teoremah}}]
Let $u\in H_0^1(\Omega)$ be a positive solution for \rife{prob} and 
$\varphi \in C^1([0,+\infty))$ be given by \rife{fih}. Observing that $\varphi (0)=0$, that $\varphi'$ is bounded and that, by \rife{estima} and since $\sigma(s) \leq 1$, we have $\varphi (s)\leq \beta$, we derive that $\varphi(u)\in H_0^1(\Omega )\cap L^\infty (\Omega )$. Therefore,
we
can take $v=\varphi(u)$ as test function in \rife{enerfin} to obtain, by using \rife{hcero},
that
\begin{eqnarray*}
\int_\Omega M (x,u)\nabla u \cdot \nabla u \, \varphi'(u)\  +\int_\Omega
h(u)|\nabla u|^2\varphi(u) \ \leq \int_\Omega f\, \varphi(u)\ .
\end{eqnarray*}
Thus, adding and subtracting $\displaystyle \frac{1}{\beta}\int_\Omega
M(x,u)\nabla u\cdot \nabla u  \, h(u)\varphi(u)\ $, we derive from \rife{coe}
and \rife{rh} that
\begin{eqnarray*}
\int_\Omega M(x,u)\nabla u\cdot \nabla u [\sigma'(u)]^2 \ &\leq& \int_\Omega
M(x,u)\nabla u\cdot \nabla u
\left[\varphi'(u)+\frac{h(u)}{\beta}\varphi(u)\right]\
\\
&&+ \int _\Omega \left[I-\frac{M\left(x,u\right)}{\beta}\right] \nabla u\cdot
\nabla u \, h(u)\varphi(u)\
\\
&\leq& \int_\Omega f\, \varphi(u)\ .
\end{eqnarray*}
Using now  \rife{coe}, \rife{estima} and the fact that $f \geq 0$, we have
\begin{equation}\label{susigma}
\alpha \int _\Omega |\nabla \sigma(u)|^2 \ = \alpha \int _\Omega |\nabla u|^2
[\sigma'(u)]^2\  \leq \int_\Omega f\, \varphi(u)\  \leq \beta \int_\Omega
f\,[\sigma(u)]^2\ .
\end{equation}
Hence, recalling (see \cite{def}) that, since $f$ belongs to
$L^q(\Omega)$ with $q>\frac N2$, and $f^+\not\equiv 0$, the first
positive eigenvalue $\lambda_1(f)$ of the eigenvalue boundary value
problem
\begin{equation*}
\begin{cases}
-\Delta u = \lambda\, f\, u & \mbox{in }\Omega,
\\
\hfill u=0 \hfill &\mbox{on } \partial \Omega,
\end{cases}
\end{equation*}
is such that
$$
\lambda_1(f)\,\int_{\Omega}\,f\,v^{2} \leq \int_\Omega\,|\nabla
v|^2, \quad \forall v \in H^{1}_{0}(\Omega),
$$
we deduce from \rife{susigma} that
$$
\alpha\, \int _\Omega |\nabla \sigma(u)|^2 \ \leq
\frac{\beta}{\lambda_{1}(f)}\, \int _\Omega |\nabla \sigma(u)|^2.
$$
Recalling the assumption  $\displaystyle \frac{\beta}{\alpha} <
\lambda_1(f)$, this implies that
$$
\int _\Omega |\nabla \sigma(u)|^2\  = 0,
$$
which yields
$$
\sigma (u)=0 \quad \mbox{for a.e. } x\in \Omega .
$$
Therefore, recalling that $\sigma(s) = 0$ if and only if $s = 0$, we have $u\equiv 0$, contradicting $u>0$ in $\Omega$: therefore,
there are no positive solutions of
\rife{prob}.
\salta{ If we suppose that \rife{infinito} is satisfied then,
{\color{red} the  above proof works to show the nonexistence of positive
solutions  $u\in H_0^1(\Omega)$ (without to impose that  $u \in
L^\infty(\Omega)$). Indeed, we have to point out that \rife{infinito} implies
that $\varphi(u)\in H_0^1(\Omega)\cap L^\infty(\Omega)$ and $\sigma(u)\in
H_0^1(\Omega)$, for every positive  $u\in H_0^1(\Omega)$. This is a direct
consequence of the the boundedness of $\varphi(s)$, $\varphi' (s)$ and
$\sigma(s)$ observed above.}}
\end{proof}

\begin{remark}\rm
Theorem~\ref{teoremah} can be extended to  more general operators. Specifically, if
$a(x,s, \varsigma)$ is a Carath\'{e}odory function such that
$$
\exists \alpha >0 \, : \, a(x, s,\varsigma )\cdot \varsigma \geq \alpha |\varsigma|^2\quad \mbox{for a.e. } x\in \Omega,
\ \forall s\in \re, \ \forall \varsigma \in \rn ,
$$
$$
\exists \beta >0 \, : \, |a(x, s,\varsigma )|\leq \beta |\varsigma| \quad \mbox{for a.e. } x\in \Omega,
\  \forall s\in \re, \  \forall \varsigma \in \rn ,
$$
and $\displaystyle 0\leq f\in L^q(\Omega )$ with $q>\frac{N}{2}$ and $f\not\equiv 0 $, then problem
\begin{equation*}
                    \label{probgen}
\begin{cases}
\displaystyle -\mbox{div\,}(a(x,u ,\D u)) + g(x,u) |\nabla u|^2 =
f & \mbox{in } \Omega, \\
\hfill u=0\hfill &\mbox{on } \partial \Omega,
\end{cases}
\end{equation*}
has no  finite energy solutions  provided
$\lambda_1(f)>\displaystyle \frac{\beta}{\alpha}$ and conditions
\rife{hcero}, \rife{contador} and
\rife{pereiro} hold.
\end{remark}

\begin{remark}
                           \label{rmka}
\rm
Let $0\leq f\in L^q (\Omega )$ with $q>\frac{N}{2}$ and $f\not\equiv 0$.
Assume \rife{coe} and  that $g(s)$ satisfies \rife{hcero}.
Observe that if $u\in H_0^1(\Omega)$ is a solution of \rife{prob}, and $R > 0$, then
$v = R u$ is a solution of
\begin{equation*}
\begin{cases}
\displaystyle -\mbox{div\,}\left( M\left(x,\frac vR\right)\nabla
v\right) + \frac1R g\left( x, \frac vR \right) |\nabla v|^2 = R\, f &
\mbox{in } \Omega,
\\
\hfill v=0\hfill & \mbox{on }\partial \Omega,
\end{cases}
\end{equation*}
with
$$
\frac{g\left( x,\displaystyle \frac{s}{R} \right)}{R} { \geq}
h_R(s)\defin \frac1R h\left(\frac sR\right).
$$
Therefore, by Theorem~\ref{teoremah}, and since $\lambda_{1}(Rf) = \lambda_{1}(f)/R$, if $h_R(s)$
satisfies conditions \rife{contador} and \rife{pereiro}, then a
necessary condition for the existence of finite energy  solutions of \rife{prob} is that $\lambda_1(f) \leq R\beta/\alpha$.
\end{remark}

In the following result, as a consequence of Theorem~\ref{teoremah} (and Remark~\ref{rmka}),
we give conditions to assure the nonexistence of
solutions of \rife{prob} for every datum $f$.

\begin{corollary} \label{teoremah2}\sl
Let $0\leq f\in L^q (\Omega )$ with $q>\frac{N}{2}$ and $f\not\equiv 0$. Assume \rife{coe} and  that $g(s)$ satisfies \rife{hcero}.
If there exists  $R_0>0$ such that the function
$h_R(s)=\displaystyle\frac1R h\left(\frac sR\right)$ satisfies
\rife{contador} and \rife{pereiro} for every $R\in (0,R_0)$, then \rife{prob} does not
have any finite energy  solution. \qed
\end{corollary}

As a consequence of the above results  we also have the following.

\begin{corollary} \label{casooptimo}\sl
Let $0\leq f\in \elle{q}$ with  $q>\frac{N}{2}$ and $f\not\equiv 0$. Suppose that \rife{coe} holds and that for some constants $s_0, \Lambda >0$ and $\gamma \geq 2$ we have
\begin{equation*}    
\frac{\Lambda}{s^\gamma} \leq g(x,s), \quad \mbox{ for a.e. } x\in
\Omega , \  \forall s\in(0,s_0].
\end{equation*}
If either
\begin{enumerate}
\item[(i)]  $\gamma >2$,
\end{enumerate}
or
\begin{enumerate}
\item[(ii)] $\gamma =2$ and $\lambda_1(f) >\frac{\beta}{\Lambda \alpha}$,
\end{enumerate}
then  \rife{prob} does not have any finite energy solution.
\end{corollary}

\begin{proof}[{\sl Proof.}]
   Consider a continuous function $h(s)$ such that
$$
  h(s)
         = \left\{
      \begin{array}{ll}
                             \dys \frac{\Lambda}{s^\gamma} & \mbox{if } 0< s \leq \dys \frac{s_0}{2} ,
                             \\
                             \dys \leq \frac{\Lambda}{s^\gamma} & \mbox{if } \dys\frac{s_0}{2}< s < s_0 ,
                             \\
                             0 &\mbox{if } s_0\leq s .
         \end{array}
                        \right.
  $$
Observing that $h_R(s)=\displaystyle \frac{\Lambda R^{\gamma-1}}{s^\gamma}$ for every $s\in (0,\frac{s_0}{2})$,  and using that $\gamma \geq 2$, we have that $h_R(s)$ is not integrable in $(0,\frac{s_0}{2})$, i.e., it satisfies \rife{contador}.

In addition, if $\gamma>2$, then $h_{R}(s)$ satisfies \rife{pereiro} for every $R>0$, so that Corollary~\ref{teoremah2} concludes the proof in this case.

On the other hand, if we assume that $\gamma =2$, then
$$
\sqrt{h_R(s)}\  {\rm e}^{\displaystyle \int_1^s \sqrt{h_R(t)}dt} =
\frac{\sqrt{\Lambda R}}{s}\  {\rm e}^{\displaystyle \int_{s_0/2}^s
\frac{\sqrt{\Lambda R}}{s} dt + \int_1^{s_0/2} \sqrt{h_R(t)}dt} = C
s^{\sqrt{\Lambda R}-1},
$$
for some $C>0$. Thus,  $h_R(s)$ satisfies  \rife{pereiro}  if and
only if
$$
\lim_{s\to 0^+}   s^{\sqrt{\Lambda R}-1} \geq 0,
$$
i.e.,  $R\geq \frac{1}{\Lambda}$.
Therefore, Remark~\ref{rmka} implies the nonexistence of solutions provided that
$\lambda_1(f) >\frac{\beta}{\Lambda \alpha}$.
\end{proof}

As a consequence of this result, we have that the first part of Theorem \ref{modelt} is proved. We are now going to prove the second part of it.

\begin{proof}[{\sl Proof of Theorem \ref{modelt}}]
We first note that if  $\gamma < 2$, then Theorem \ref{th1}
guarantees the existence of a solution. Conversely, if $\gamma >
2$ or if  $\gamma = 2$ and $\|f\|_{\lio}$ is large enough,
Theorem \ref{teoremah} applies and no solutions exist for
\rife{model}.

On the other hand, if $f\in L^\infty(\Omega )$ and \eqref{ipnnex}
holds,
we recall that  existence and uniqueness of a solution
$u_{n}$ in $\huz\cap C(\overline{\Omega} )$ for
\be \label{EQ}
\begin{cases}
\dys -\Delta u_n +\frac{|\D u_n |^2}{ \left(u_n + \frac{1}{n}\right)^{\gamma} } = f & \mbox{in } \Omega, \\
\hfill  u_n = 0 \hfill & \mbox{on } \partial \Omega,
\end{cases}
\ee (with $\gamma \geq 2$)    follows by the results of \cite{BMP}
(existence) and \cite{AS} (uniqueness). Taking $u_{n}$,
$G_{k}(u_{n})$, and $T_{\eps}(u_{n})/\eps$ as test functions and
working as in Lemma~\ref{lemaprop} (1.), it is easy to see that
$u_{n}$ is bounded in $\huz$ and in $L^\infty (\Omega )$, and that
there exists $C>0$ (independent on $n$) satisfying
$$
\int_{\Omega}\,\frac{|\D u_n |^2}{ \left(u_n + \frac{1}{n}\right)^{\gamma} } \
\leq C.
$$
Therefore, up to subsequences, there exists a nonnegative bounded
Radon measure ${\nu}$ such that
$$
\frac{|\D u_n |^2}{ \left(u_n + \frac{1}{n}\right)^{\gamma} }
\mbox{ converges to } {\nu} \mbox{ in the weak-$\ast$
topology of measures.}
$$
Since $u_{n}$ is bounded in $\huz$ then it converges, up to
subsequences, to some function $u$ weakly in $\huz$, strongly in
$\elle2$, and almost everywhere in $\Omega$. Moreover, since $f -
\frac{|\D u_n |^2}{ \left(u_n + \frac{1}{n}\right)^{\gamma} }$ is
bounded in $\elle1$, the result of \cite{BM} yields that (up again
to subsequences) $\D u_n$ converges to $\D u $ almost everywhere in
$\Omega$. Then we have, by Fatou lemma, that $\frac{|\D u |^2}{
u^{\gamma} }\chi_{\{u>0\}}$ belongs to $L^1 (\Omega)$, and that
$$
\nu =  \frac{|\D u |^2}{   u^{\gamma} }\chi_{\{u>0\}}
+{\nu_0},
$$
where ${\nu_0}$ is a nonnegative bounded Radon measure on $\Omega$.
Therefore, $u\in\huz$ is a finite energy solution of
$$
\begin{cases}
\dys -\Delta u +\frac{|\D u |^2}{ u^{\gamma} }\chi_{\{u>0\}} = f - {\nu_0} & \mbox{in $\Omega$,} \\
\hfill u = 0 \hfill & \mbox{on $\partial\Omega$.}
\end{cases}
$$
Note also that since $u_{n+1}$ is a subsolution for \rife{EQ}, we
can apply the comparison principle of \cite{AS} so
that, for  every $x \in \Omega$, we have
\begin{equation*}
u_{n}(x) \geq u_{n+1}(x) \geq \ldots \geq u(x),
\end{equation*}
 and thus we can assume that $u_n(x)$ is converging to $u(x)$ for
 every $x\in\Omega$. We claim that $u\equiv 0$, so that $u_{n}$ converges to zero in $\elle2$.
 Indeed,
we divide the proof of this assertion in two steps:

{\bf Step 1.} The case in which $\Omega$ is a ball of radius $R>0$, $\Omega = B_R$, and
$f = T>0$ is a constant.

{\bf Step 2.} The general case.

\noindent {\bf Step 1.} Assume that $\Omega = B_R$ and $f = T>0$ is
a constant. In this case, \rife{ipnnex} means that, if $\gamma =2$,
then the first eigenvalue  $\la_{1}^{B_{R}}(T)$ of the Laplacian
operator with weight $T$ in $B_R$ is greater than one, i.e.,
$\la_{1}^{B_{R}}(T)> 1$.
We first observe that $u$ is radially
symmetric (and thus continuous for $|x|\not= 0$). Indeed, if we
define
$$
\psi_{n}(s)=\int_0^s e^{-H_{n}(t)} d t,\quad \mbox{where}\quad
H_{n}(t)=
\frac{n^{\gamma-1}}{\gamma-1}\left[ 1 - (1+nt)^{1-\gamma}\right]
$$ and we set $ v_n=
\psi_{n}(u_n)$, it is easy to check that $v_n$ is the unique
solution of
$$
\begin{cases}
-\Delta v_n =T {\rm e}^{-H_{n}(\psi_{n}^{-1} (v_n))} \quad & \mbox{in }\, B_R\\
\hfill v_n=0 \hfill & \mbox{on }\,\partial B_R\,.
\end{cases}
$$
Since the nonlinearity $0\leq {\rm e}^{-H_{n}(\psi_{n}^{-1} (s))}$
is $C^1$, we can apply the result of  Gidas, Ni and Nirenberg (see
\cite{gnn}) in order to deduce  that $v_n$ is radially symmetric
(hence $v_{n} = v_{n}(r)$), monotone decreasing with respect to $r$
and such that $v_n'(0)=0$. Since $\psi_{n}$ and $H_n$ are smooth and
increasing, the functions $u_n$ have the same properties as $v_n$.
Passing to the limit with respect to $n$ we deduce that $u$ is
radially symmetric and monotone nonincreasing.

We argue by contradiction assuming that $u$ is not identically zero.
In this case, using that $u(r)$ is nonincreasing in $(0,R)$,
$$
r_{1} = \inf \{0 < r \leq R : u(r) = 0\} > 0,
$$
and then
$$
u \geq c_{\eps}:= u(r_1-\eps)   \quad \mbox{in } B_{r_1 - \eps}\,.
$$
Therefore, repeating the proof of Theorem \ref{teorema}, we prove that
$$
\lim_{n \to +\infty}\,\frac{|\D u_n |^2}{ \left(u_n + \frac{1}{n}\right)^{\gamma} } = \frac{|\nabla u|^{2}}{u^{\gamma}} \quad \mbox{strongly in $L^{1}_{{\rm loc}}(B_{r_{1}})$,}
$$
so that $\nu_{0}$ is zero on $B_{r_{1}}$ and, by  the continuity of
$u$ for $r\not= 0$,  $u$ is a solution of
$$
\begin{cases}
\dys -\Delta u  +\frac{|\D u  |^2}{  u^{\gamma} } = T & \mbox{in } B_{r_1}, \\
\hfill  u  = 0 \hfill & \mbox{on } \partial B_{r_1},
\end{cases}
$$
and this contradicts the result of Theorem \ref{teoremah}  (note
that, if $\gamma = 2$, we have $\la_{1}^{B_{r_{1}}}(T) >
\la_{1}^{B_{R}}(T) > 1$). Therefore
$u \equiv 0$.

\smallskip

\noindent {\bf Step 2.} $\Omega$ is an open set and $f$ is
nonnegative and belongs to $\lio$.

By \rife{ipnnex}, we can fix $R > \mbox{diam}\, \Omega$ with
$\lambda_1> \|f\|_{L^\infty(\Omega)} R^2$ provided that $\gamma=2$.
Let $v_n$ be also the solution of
\begin{equation*}
\begin{cases}
\dys -\Delta v_n +\frac{|\D v_n |^2}{\left(v_{n} + \frac{1}{n}\right)^{\gamma} } = \|f\|_{\lio} & \mbox{in } B_R \\
\hfill v_n =0 \hfill &\mbox{on } \partial B_R.
\end{cases}
\end{equation*}

By definition of $\mbox{diam}\, \Omega$,  we have $\overline{\Omega}
\subset B_R$.
Our aim is to prove that
${v}_{n}$ is a supersolution for \rife{EQ}. Indeed, let $ 0\leq\psi
\in  C^{\infty}_{0} (\Omega)$ and we use it as test function in the
formulation of $v_n$. Thus
$$
\dys \int_{B_R} \D v_n \cdot \D \psi +\int_{B_R} \frac{|\D v_n|^2}{(\frac1n + u_n)^{\gamma} } \psi =\int_{B_R} \|f\|_{L^{\infty} (B_R)} \psi\,,
$$
and since the support of $\psi $ is contained in $\Omega$ we deduce
$$
\dys \intO \D v_n \cdot \D \psi +
\intO \frac{|\D v_n|^2}{(\frac1n + u_n)^{\gamma} } \psi
=\intO \|f\|_{L^{\infty} (B_R)} \psi \geq \intO f \psi
$$
for every nonnegative $\psi$ in $H^1_0 (\Omega)\cap L^{\infty}
(\Omega)$ (by an easy density argument). Using again the comparison
principle of \cite{AS}, $ {u}_{n}\leq v_n$ in $\Omega$. Now,
observing that by  the choice of $R$, if $\gamma =2$, we have
\begin{equation}
\label{ddd} \la_{1}^{B_{R}}(\|f\|_{\elle{\infty}})=
\frac{\lambda_{1}}{R^{2}\|f\|_{\elle{\infty}}}> 1,
\end{equation}
we are able to apply the previous Step 1, so that
$v_n$ tends to~0 strongly in $L^{2}(B_R)$, which implies that $
{u}_{n}$ tends to zero in $L^{2}(\Omega)$ and the claim has been
proved.

\medskip
Finally, we conclude the proof by taking $u_{n}$ as test function in
\rife{EQ} and dropping the nonnegative quadratic term to deduce that
the convergence to zero is strong in $\huz$; using this fact in the
weak formulation of \rife{EQ} then yields that $\nu = f$, as
desired.

\end{proof}

\begin{remark}\rm
Let us emphasize that the condition $\|f\|_{\elle{\infty}} \leq
\frac{\lambda_1}{(\mbox{diam\,}\Omega)^2}$ imposed in assumption
\rife{ipnnex} for the case $\ga=2$ is not optimal. We use it for the
sake of simplicity. However,  as shown in the proof of
Theorem~\ref{modelt} (see \eqref{ddd}), a sharper condition can be
used in this case.

More precisely, if  we consider the Chebyshev radius $R(\Omega)$ of $\overline{\Omega}$, i.e.
the greatest lower bound of the radii of all balls containing $\overline{\Omega}$, then the the result of Theorem \ref{modelt} with $\ga=2$ holds  provided that
$$
\|f\|_{\elle{\infty}} < \frac{\lambda_{1}}{R(\Omega)^{2}}.
$$
\end{remark}

\appendix
\setcounter{equation}{0}
\section{Local a priori estimates and large solutions}

We devote this appendix to recall some results concerning the following
equation
\begin{equation}\label{ls}
-\div (a(x,u ,\D u)) +B(x,u) =F (x, \D u), \quad x\in \Omega ,
\end{equation}
\noindent where $F(x,\varsigma)$ and $a(x,s, \varsigma), B(x,s)$
are Carath\'eodory functions. Suppose that there exist constants $\beta\geq
\alpha >0$ such that
\begin{equation}\label{a1}
a(x, s,\varsigma )\cdot \varsigma \geq \alpha |\varsigma|^2,
\end{equation}
\begin{equation}\label{a2}
|a(x, s,\varsigma )|\leq \beta |\varsigma|,
\end{equation}
\begin{equation}\label{a3}
(a(x, s,\varsigma )-a(x, s,\eta ))\cdot(\varsigma -\eta)>0 ,
\end{equation}
for a.e. $x\in \Omega$, for every $s\in \re$ and for every $ \varsigma$, $\eta \in
\rn$,   $\varsigma \neq \eta$.

We suppose that  $F(x,\varsigma)$ satisfies
\begin{equation}\label{effe}
|F(x,\varsigma)|\leq F_0 (x) + p_0 |\varsigma|, \, \mbox{a.e. }x\in
\Omega, \, \forall \varsigma\in \mathbb{R}^N,
\end{equation}
where $F_0$  belongs to $L_{\mbox{\tiny loc}}^q (\Omega )$ with
$q>\frac N2$ and $p_0>0$. We  also suppose that  there exists a
continuous nonnegative function $b:[0,+\infty) \to [0,+\infty)$ such
that
\begin{equation}\label{g1}
\begin{array}{l}
b(s) \mbox{ is increasing and satisfies the Keller-Osserman condition \rife{k-o}, } \\
\dys b(s)/s \mbox{ is
nondecreasing for large } s,
\\ \mbox{and for every }\omega \subset \subset \Omega   \mbox{ there exists
$m_{\omega}>0$ such that }
\\
B(x,s)\geq m_{\omega} b(s)\geq 0, \mbox{for a.e. } x\in \omega,
\mbox{ for every } s\in \re^+.
\end{array}
\end{equation}
Then the subsolutions of equation \rife{ls} are uniformly bounded from above in $\omega\subset \subset \Omega$. This result is essentially contained in \cite{LEO}.
\begin{theorem}    \label{thleoni}\sl
Suppose that $a(x,s,\varsigma)$ satisfies \rife{a1}--\rife{a3},  $B(x,s)$ satisfies
\rife{g1} and assume that $F_0 \in L_{\mbox{\tiny loc}}^q (\Omega )$,
$q>\frac{N}{2}$.  Then, for every $\omega\subset \subset \Omega$ there
exists $C_\omega >0 $ such that  any  distributional subsolution $u\in H_{\mbox{\tiny loc}}^1(\Omega)$ of (\ref{ls})
 such that $u^+\in L^\infty_{\rm loc} (\Omega)$ and $B(x,u^+)\in
L^1_{\rm loc} (\Omega)$ satisfies
$$
u(x)\leq C_{\omega},\quad \forall x\in \omega\,.
$$
\end{theorem}

In order to prove this theorem, we need the following two lemmas.

\begin{lemma}[Lemma~3.3 of \cite{LEOBEN}]\label{leoni}\sl
Let $b:[0,+\infty) \to [0,+\infty)$ be a continuous  function,
satisfying the Keller-Osserman condition \rife{k-o}, such that
$\frac{b(s)}{s}$ is nondecreasing for large $s$. Then, for any $C>0$
and $\gamma\geq 0$, there exist a positive constant
$\Gamma$ and  a smooth function $\vp:[0,1]\longrightarrow [0,1]$
 ($\Gamma$ and $\varphi$ depending only on $b$, $C$ and
$\ga$), with $\vp(0)=\vp'(0)=0$, $\vp(1)=1$ and
$\varphi(s)>0$ for every $s>0$, satisfying
\begin{equation*}
\dys t^{\gamma +1 }\, \frac{\vp'(\tau)^2}{\vp(\tau)}\, \leq \,
\frac{1}{C}\, t^{\ga}\, b(t)\, \vp(\tau) +
\Gamma,\qquad \forall \tau \in (0,1], \, \forall t\geq 0.
\end{equation*}

\end{lemma}

\begin{remarks}\rm \begin{enumerate}

\item In Lemma~3.3 of \cite{LEOBEN} it is imposed  that   $b(s)$ is increasing, $b(0)=0$, and the function
$\frac{b(s)}{s}$ is nondecreasing  in $\mathbb{ R}^+$. However, it is easy to
see that the proof (see also \cite{LEO}) works by using the weaker assumptions of
Lemma~\ref{leoni}.
\item In addition, also in \cite{LEO},  the Keller-Osserman condition is replaced by the following one:
$$
\displaystyle \int^{+\infty} \frac{ds}{\sqrt{s b(s)}} <+ \infty .
$$
Note that, as a consequence of the monotonicity of $b(s)$ for large $s$,
the above assumption is equivalent to \rife{k-o}.
\end{enumerate}
\end{remarks}

Let us recall a local version of a classical result by Stampacchia we will use
in the following.

\begin{lemma}[\cite{S}] \label{stampa}\sl
Let $\tau (j,\rho)\ :[0,+\infty)\times[0,R_0) \longrightarrow \mathbb R $ be
a function such that $\tau (\cdot,\rho) $ is nonincreasing and $\tau
(j,\cdot) $ is nondecreasing. Moreover, suppose that there exist $K_0 >0$, $\mu>1,$ and $C,\nu,\ga >0$ satisfying
$$
\tau (j,\rho) \leq C \frac{ \tau (k,R)^\mu}{(j-k)^\nu (R-\rho)^\ga} ,\quad
\forall j>k>K_0 ,\, \forall \ 0< \rho < R < R_0.
$$
Then for every $  \delta \in (0,1)$, there exists $ d>0$ such that:
$$
\tau (K_0+ d,(1-\delta)R_0) =0\,,
$$
where $d^\nu = 2^{(\nu+\gamma)\frac{\mu}{\mu-1}} C
\frac{\left(\tau(K_0,R_0)\right)^{\mu-1}}{\delta^\gamma R_0^\gamma}$.
\end{lemma}

\begin{proof}[{\sl Idea of the Proof of Theorem \ref{thleoni}}]
The proof of this result is essentially contained in \cite{LEO}, but
for the convenience of the reader, we include here the proof of the
exact result that we have used in the proof of
Proposition~\ref{prop2} and in the proof of Theorem~\ref{notdiv}.

Actually we deal with equation
\begin{equation*}
-{\rm div} (\widetilde M(x,u)\nabla u )+  P(x)\cdot \D u +f(x) b(u)=
0 , \ \ {\rm in}\ \Omega\,,
\end{equation*}
where $\widetilde M(x,s)$ satisfies \rife{coe}, $P(x)$ is a bounded
vector field, $b(s)$ { is increasing and} satisfies the
Keller-Osserman condition \rife{k-o}, $\frac{b(s)}{s}$ is
nondecreasing for $s$ large, and $f$ satisfies $\rife{cona}$.
Consequently all the assumptions of the theorem are satisfied. We
remind that the above assumptions are satisfied by the functions
$\widetilde M(x,s)$, $b(s)$ and $f(x)$, appearing in Proposition
\ref{prop2} as well as by the functions $M(x), P(x), b(s)$ and
$f(x)$ appearing in Theorem~\ref{notdiv}.

Suppose now that $u^+\in L^\infty_{\mbox{\tiny loc}}(\Omega)$ and
$b(u^+)\in L^1_{\mbox{\tiny loc}}(\Omega)$. We set
$\omega\subset\subset\omega'\subset\subset\Omega$ and a cut-off
function $ \eta (x)$ such that $0\leq \eta\leq 1$ and
\begin{equation}\label{tha11}
\eta (x) = \left\{
\begin{array}{l}
1, \quad x\in\omega,   \\
    0, \quad x\in \Omega \backslash \omega'.
\end{array}\right.
\end{equation}

\noindent We denote $p_0={\|P(x)\|_{(\elle{\infty})^N}}$ and we fix
$\sigma>\frac{2{p_0}}{\alpha}$ and the constants $C, k_0$ such that
$$
\frac{ \|\nabla\eta\|^{2}_{\elle{\infty}}}{8 \sigma^2}
\left[\frac{\beta^2}{\alpha} +\left(\alpha -\frac{{p_0}}{\sigma}
\right)\right]\frac1C + \frac{p_0}{4\sigma b(k_0)}\leq
\frac{m_{\omega'}(f)}{2\sigma},
$$
and we also consider the function $\varphi$ given by
Lemma~\ref{leoni} with $\gamma =1$ and this constant $C$. Note that
if $\xi=\sqrt{\varphi(\eta)}$, then { $u\xi^2 = u\varphi (\eta )\in
H_{0}^1(\Omega)$} and
\begin{equation}\label{tha12}
\nabla (u \xi ^2) = \left\{
\begin{array}{cl}
\xi^2 \nabla u + 2\xi u \nabla \xi, &\mbox{if } \xi (x)>0,\\
0, &\mbox{if } \xi (x)=0 ,
\end{array}\right.
\end{equation}
a.e. in $\Omega$.  Moreover $f(x) (e^{2\sigma G_k(u^+)} -1)
b(u^+)\in L^1_{\mbox{\tiny loc}}(\Omega)$ and consequently  $v=
\frac{1}{2\sigma} (e^{2\sigma G_k(u^+)} -1) \xi^2$,
$\sigma>\frac{2{p_0}}{\alpha} $ is an admissible test function.
Using this test function as well as \rife{coe}, \rife{cona} and
Young's inequality we deduce that

$$
\begin{array}{lll}
 &{\alpha} & \dys  \int_{\omega'} |\nabla G_k(u^+)|^2 e^{2\sigma G_k(u^+)} \xi{  ^2}
+\frac{m_{\omega'}(f)}{2\sigma} \int_{\omega'} b(u^+)(e^{2\sigma G_k(u^+)} -1) \xi^2
\\[1.5 ex]
& \leq &\dys  
 \frac{ \beta}{{\sigma}}
\int_{\omega'} {  |\nabla \xi | | \D G_k(u^+)| } (e^{2\sigma
G_k(u^+)} -1) \xi 
+\, {p_0}  \int_{\omega'} | \D
G_k(u^+)| v
 \,
\\[1.5 ex]
&\leq & \dys  \frac{\alpha}2 \int_{\omega'} |\nabla G_k(u^+)|^2 e^{2\sigma
G_k(u^+)} \xi{^2}+ \frac{\beta^2}{2\alpha
\sigma^2}\int_{\omega'}\frac{|\nabla \xi |^2 (e^{2\sigma G_k(u^+)}
-1)^2}{e^{2\sigma G_k(u^+)}}
\\[1.5 ex]
&&\dys   + \frac{{p_0} }{{2\sigma}}\int_{\omega'} | \D G_k(u^+)|^2
(e^{2\sigma G_k(u^+)} -1) \xi^2 + \frac{{p_0}
}{{4\sigma}}\int_{\omega'} (e^{2\sigma G_k(u^+)} -1) \xi^2
 \,
\\[1.5 ex]
&\leq & \dys    \frac{\sigma \alpha+p_0}{2\sigma}
\int_{\omega'} |\nabla G_k(u^+)|^2 e^{2\sigma G_k(u^+)} \xi{^2} +
\frac{\beta^2}{2\alpha \sigma^2}\int_{\omega'}|\nabla \xi |^2
(e^{2\sigma G_k(u^+)} -1)^2
\\[1.5 ex]
 && \dys + \frac{{p_0} }{{4\sigma}}\int_{\omega'} (e^{2\sigma G_k(u^+)}
-1) \xi^2\,.\end{array}
$$

In other words,
$$
\begin{array}{c}
\displaystyle \frac{1}{4\sigma^2} \left( \alpha-
\frac{{p_0}}{\sigma}\right) \int_{\omega'} |\nabla [ (e^{\sigma
G_k(u^+)}-1) \xi] |^2 +\frac{m_{\omega'}(f)}{2\sigma} \int_{\omega'}
b(u^+)(e^{2\sigma G_k(u^+)} -1) \xi^2
\\\\
\displaystyle \, \leq \left[\frac{\beta^2}{2\alpha \sigma^2}
+\frac1{2\sigma^2} \left(\alpha -\frac{{p_0}}{\sigma} \right)\right]
\int_{\omega'} { |\nabla \xi |^2 }  (e^{2\sigma G_k(u^+)}-1)^2 +
 \frac{{p_0}}{2} \int_{\omega'} v.
 \end{array}
$$

Applying Lemma \ref{leoni} with $\ga=1$, together with the
monotonicity of $b(s)$ we get
$$
\frac{1}{4\sigma^2} \left( \alpha- \frac{{p_0}}{\sigma}\right)
\int_{\omega'} |\nabla [ (e^{\sigma G_k(u^+)}-1) \xi] |^2
$$
$$ \leq \dys
 \Gamma \frac{
\|\nabla\eta\|^{2}_{\elle{\infty}}}{8 \sigma^2}
\left[\frac{\beta^2}{\alpha} +\left(\alpha -\frac{{p_0}}{\sigma}
\right)\right] \mbox{meas\,} \{x\in\omega' : u(x)\geq k\}\, ,
$$
for every $k>k_0 (m_{\omega'}(f), {p_0})$. We deduce by Sobolev's
inequality that
\begin{eqnarray*}
\dys \left(\int_{\omega} |(e^{\sigma G_k(u^+)}-1) { \xi}|^{2*} \right)^{\frac{2}{2*}} \leq
 {
    C_0 } \mbox{ meas\,}\{x\in\omega' : u(x)\geq k\} ,
\end{eqnarray*}
where $C_0 =  \mathcal{S}^2  \Gamma
\frac{\|\nabla\eta\|^{2}_{\elle{\infty}}}{8\sigma^2}
\left[\frac{\beta^2}{\alpha} +\left(\alpha -\frac{{p_0}}{\sigma}
\right)\right]$. Hence, using that $e^t-1\geq t$, for every $t\geq
0$, and that $G_k(s) \geq j-k $  for $s\geq j>k$ we derive that
\begin{equation} \label{Alonso}
(j-k)^{2} \mbox{ meas\,} \{x\in\omega : u(x)\geq j\}^{\frac{2}{2^*}}
 \leq
 \frac{ C_0}{\sigma} \mbox{ meas\,}
\{x\in\omega' : u(x)\geq k\} \, .
\end{equation}

\noindent  Now, if $\omega\subset\subset\Omega$ is fixed, we
consider $R=\mbox{dist\,} (\omega , \partial \Omega)/2$, the set
$$
\omega_r = \{ x\in \Omega : \mbox{dist\,} (x,\omega ) <r\} \subset\subset \Omega
$$
and the function
$$
\tau (k,r) =  \mbox{meas\,}  \{x\in\omega_r : u(x)\geq k\},
$$
for every $r\in (0,R]$ and $k>0$. Taking $\omega = \omega_r$ and
$\omega'= \omega_R$ in \rife{Alonso} and choosing $\eta$ such that
$\|\nabla\eta\|_{\elle{\infty}}\leq \frac{c}{R-r}$, we obtain
$$
(j-k)^{2} \tau (j,r)^{2/2^*}
                     \leq
                     {   c_1}
\frac{\tau (k,R)}{(R-r)^2} \,
$$
for some $c_1>0$ and the proof is concluded by applying
Lemma~\ref{stampa}.
\end{proof}

\begin{remarks}  \label{rmke}
{\rm
\begin{enumerate}
\item { We remark explicitly that in the above proof the constant
$\Gamma$ obtained by applying  Lemma~\ref{leoni} depends  on $m_{\omega'}(f)$.
In particular, since $m_{\omega'}(f)\leq m_{\omega}(f)$ for $\omega\subset
\subset \omega'$, if $m_{\omega}(f)$ tends to zero, then this constant $\Gamma$
and hence  the a priori estimate $C_\omega$ given by Theorem~\ref{thleoni}
diverge to $+\infty$. }

\item By adding a condition on the function $b(s)$ for negative $s$ and
using similar ideas to these ones in the above proof, it is possible
to give also a priori estimates of the whole $L^{\infty}$ norm of
the solution in every  compact subset $\omega$ of $\Omega$. More
precisely, if, in addition to the hypotheses of
Theorem~\ref{thleoni}, we  strengthen \eqref{g1}  by imposing that
  \begin{equation}\label{g1prima}
\begin{array}{l}
b(s) \mbox{ is increasing and satisfies \rife{k-o}, } b(s)/s \mbox{ is
nondecreasing},
\\
\mbox{ for large  $s$ and for every }\omega \subset \subset \Omega \mbox{ there
exists $m_{\omega}>0$ such that:}
\\
\forall s\in \re, \quad B(x,s)\, \mbox{sign} \, s\geq m_{\omega} b(|s|)\geq 0
\quad \mbox{for a.e. } x\in \omega ,
\end{array}
\end{equation}
then for every $\omega \subset \subset \Omega$ there exists $C_\omega >0 $ such
that
$$
|u(x)|\leq C_{\omega},\quad \forall x\in \omega\,.
$$
\end{enumerate}
}
\end{remarks}

Theorem~\ref{thleoni} is an extension to quasilinear equations of the
well-known local a priori estimate of Keller \cite{K} and Osserman \cite{o}
(see also \cite{BaM}, \cite{MV}, \cite{MV2}, \cite{jlv}, \cite{Ve} and
the references cited therein)  for semilinear operators. This {  semilinear a
priori estimate} was the crucial tool in order to prove the existence of a
large solution, i.e., a solution $u$ of the semilinear equation satisfying
$u=+\infty$ at $\partial \Omega$ in the sense that
\begin{equation*}
\lim_{\mbox{{\tiny dist}} (x,\partial \Omega ) \to 0} u(x) =+\infty.
\end{equation*}
Thus, it is natural to ask whether it is also possible to prove the existence
of a large solution for \rife{ls}. Clearly, in this nonlinear framework we have
to specify the meaning we give to \lq\lq infinity\rq\rq  at $\partial \Omega$,
since it has no sense pointwise. Actually we will assume such a condition in a
weak sense, through a condition on the trace on the boundary of the truncation
of the solution. Specifically we consider the following equation
\begin{equation}\label{ls2}
-\div (a(x,u ,\D u)) +B(x,u) =F (x), \quad x\in \Omega.
\end{equation}

\begin{definition}\label{defi}\rm
An a.e. finite function $u(x)$ such that $ T_k (u)\in H^1 (\Omega)$ $ \forall
k>0$   is  a {\sl distributional large solution} for \rife
{ls2} with $F\in L^1_{\rm loc} (\Omega)$, if:\\
i) $|a(x,u,\D u)| \in L^1_{\rm loc} (\Omega)$, $B(x,u) \in L^1_{\rm loc} (\Omega)$;\\
ii)
$$
\intO a(x,u,\D u)\cdot \D \varphi \,  + \intO B(x,u) \varphi \, = \intO F\,
\varphi \,  ,
\qquad \forall \varphi\in C^{\infty}_c (\Omega)\,;$$\\
iii) $\forall k>0$, $k-T_k (u) \in H^1_0 (\Omega)$.
\end{definition}

\begin{remark}
{\rm In the above definition, iii) has the meaning of \lq\lq infinity at
$\partial \Omega$\rq\rq. We mention that this definition of  explosive boundary
condition has already been introduced in \cite{tom}, for a different class of
nonlinear elliptic equations involving nonlinear \lq\lq coercive\rq\rq{}
gradient terms. }
\end{remark}

We conclude by observing that even if not explicitly written in   \cite{LEO},
all the estimates that we need in order to prove the existence of large
solutions for \rife{ls2} have been proved and thus we have the following
result.
\begin{theorem}\label{existls}\sl
Suppose that $a(x,s, \varsigma)$ and $B(x,s)$ satisfy  \rife{a1}, \rife{a2},
\rife{a3}, \rife{g1prima} and
\begin{equation}\label{g2}
\sup_{|s|\leq k }
|B(x,s)|\in L^1  (\Omega),\quad \forall k>0.
\end{equation}
Assume also that $F\in L^1_{\rm loc} (\Omega)$ with $F^- \in L^1 (\Omega)$.
Then there exists a  distributional large solution for $\rife{ls2}$.
\end{theorem}

\proof[{\sl  Proof.}] We consider the following sequence of
problems
$$
\left\{
\begin{array}{cl}
- \mbox{\rm div } a(x,u_n,\D u_n  ) + B(x,u_n)  = F_n & \mbox{in }\ \Omega, \\
u_n - n \in \huz, &
\end{array}\right.
$$
where $ F_n =T_n (F)$. Since $B(x,s+n)s\geq 0$ for large $|s|$, the
existence of a weak solution $u_n\in H^1(\Omega)\cap \lio$ is a consequence of
\cite{b6} (Theorem 6.1), i.e. $u_n-n\in H_0^1(\Omega)$ and it satisfies
\begin{equation}\label{wap}
\begin{array}{c}
\intO a(x,u_n, \D u_n )\cdot \D v  \,
+ \intO B(x,u_n)v \,
 =\intO  F_n v \,
 , \,  \forall v\in H_0^1 (\Omega)\cap \lio .
\end{array}
\end{equation}
Observing that for any  $n\geq k$, $k-T_k (u_n)\in H_0^1(\Omega)\cap\lio$,  we
can choose $v=k-T_k (u_n)$ as test function in \rife{wap} and we  obtain,
$$
- \intO a(x,u_n,\D u_n  )\cdot  \D T_k (u_n)    \,
 + \intO B(x,u_n) [k-T_k (u_n)]  \,
= \intO  F_n [k-T_k (u_n)] \,.
$$
Using  \rife{a1}, and  \rife{g1} and  \rife{g2} we have:
\begin{equation*}
\al\intO | \D T_k (u_n)|^2  \,  \leq 2k\intO \sup_{|s|\leq k}|B(x,s)| \, + 2k
\| F_n^- \|_{\elle1}.
\end{equation*}
Thus, for every $k\in \na$, we can now extract a subsequence (not relabeled) of
$\left\{T_k(u_n)\right\}_{n\in \na}$ that weakly converges in $H^1 (\Omega)$
and, by Rellich theorem, strongly in $\elle{2}$.

Now, consider any sets $\omega \subset \subset \omega' \subset \subset \Omega$,
a cut-off function $\eta(x)$ chosen as in \rife{tha11} and $\xi=
\sqrt{\varphi(\eta)}$.  Arguing as in \rife{tha12}, we deduce that   $v= T_k
(u_n \xi^2) $ is an admissible test function for \rife{wap}. Thus we have  
$$
\int_{A_k} a(x,u_n, \D u_n )\cdot \D[ u_n \xi^2]   \,
+ \intO B(x,u_n)T_k (u_n \xi^2)   \,
\leq k\|F\|_{L^1(\omega')}\,,
$$
where  $A_k = \left\{
             x\in \Omega :
             |u_n|\xi^2\leq k
             \mbox{ and }
             \xi(x) >0
             \right\}$, 
and so, using \rife{a1} and \rife{g1prima}, we get 
\begin{equation*}
\begin{array}{c}
\dys  \alpha \int_{A_k} | \D u_n|^2  \xi^2   \,
 + m_{\omega'} \int_{A_k} |b(u_n)| T_k (u_n \xi^2)  \\[1.5 ex]
\dys    \leq   k\|F\|_{L^1(\omega')} +     2\beta \int_{A_k  } |  \D u_n || \D
\xi| \, u_n \xi  \,  .
  \end{array}
\end{equation*}
By applying Young inequality, \rife{a2} and Lemma \ref{leoni} (with  $\gamma=1$ and  for any fixed $C>\frac{\alpha^2+4 \beta^2}{8\alpha m_{\omega'} }
\|\D \eta \|_{L^{\infty} (\omega')}$ and taking into account
{ Remarks~\ref{rmke}-2}) we deduce that there exists
$c>0$ such that
\begin{eqnarray*}
\intO |\D T_k( u_n\xi^2)|^2   \, \leq  c  (k+1).
\end{eqnarray*}
Then, using that $\xi=1$ in $\omega$, by Lemmas~4.1 and 4.2 of \cite{b6} it
follows that $u_n$ and $|\D u_n|$ are bounded respectively in
$\mathcal{M}^{\frac{N}{N-2}}(\omega)$ and $\mathcal{M}^{\frac{N}{N-1}}
(\omega)$, for any $\omega\subset\subset\Omega$. Combining this information
with the strong convergence of $T_k (u_n)$ in $L^2 (\Omega)$ we deduce that
$u_n$ is a Cauchy sequence in measure and so, up to subsequences (not relabeled),
it converges for a.e. $x\in \Omega$ to a function $u  \in W^{1,q}_{\rm loc}
(\Omega)$. This, in particular, implies that
\begin{equation*}
\lim_{n \to +\infty}
k-T_k (u_n) =
k-T_k (u) \quad \mbox{weakly in } \huz\,,
\end{equation*}
i.e. $u$ satisfies the boundary condition.

On the other hand, we prove that the lower order term is bounded in
$L^1_{\mbox{\tiny loc}} (\Omega )$; indeed,  if, for $\varepsilon>0$, we take
$v=\frac{1}{\eps} T_{\eps} (u_n)\xi$ as test function in \rife{wap} (as
before, such a function it is admissible). Thus, by \rife{a1}, \rife{a2}, and  dropping positive
terms, we get
$$
\begin{array}{c}
 \intO B(x,u_n) \frac{ T_{\eps} (u_n)  }{\eps} \xi  \,
 \leq
\|F\|_{L^1 (\omega')}  +  \beta \|\D \xi \|_{L^{\infty} (\omega')}
\int_{\omega'} |\D u_n| \, .
\end{array}
$$
Since  the right hand side is bounded being $\{|\D u_n|\}$
bounded in ${\mathcal M}^{\frac{N}{N-1}}_{\rm loc} (\Omega)$ and $F\in \loc1$,
letting $\eps\to0$, we deduce by Fatou lemma that there exists $c_{\omega}>0$ such that
$$
\dys  \int_{\omega} | B (x,u_n)| \leq c_{\omega}.
$$
On the other hand,  choosing $v=T_1 (G_h (u_n \xi^2))$ as test function, where
$\xi^2= {\varphi (\eta)}$  we have, by using \rife{a1}, \rife{a2},
\rife{g1prima} and  \rife{g2},
$$
\frac{\alpha}{2} \int_{h\leq |u_n\xi^2|\leq h+1} | \D  u_n|^2 \xi^2
+\frac12 \intO B(x,u_n)T_1 (G_h (u_n \xi^2))
$$
$$
\leq  \int_{ \omega'\cap  \{  u_n\xi^2 \geq h\}} | F_n |+ \frac{2\beta^2}{\alpha} \|\D \eta\|^2_{\lio} \mbox{ meas} \{ x\in \omega'\,:\, \xi^2 |u_n |\geq h \} .
$$
By the strong compactness of  $\{F_n\}$  in $L^1 (\omega')$ and the local
uniform estimate of $\{u_n\}_{n\in \na}$ in $\mathcal{M}^{\frac{N}{N-2}}_{\rm
loc} (\Omega)$, we derive then that
$$
\lim_{h\to +\infty} \sup_{n\in \na } \int_{\{ x\in \omega : \, |u_n |\geq h \}}
|B(x,u_n)| \,  =0.
$$
As a consequence of   Vitali theorem we deduce that
$\{|B(x,u_n)|\}_{n\in \na}$ is strongly compact in $L^1 (\omega')$, where
$\omega' \subset \subset \Omega$ is arbitrary.
Moreover, since the lower order term is bounded in $L^1_{\rm loc} (\Omega)$, we
can apply Lemma~1 in \cite{bga} in order to prove that $\D u_n$ converges to
$\D u$ a.e. in $\Omega$. This, and the weak convergence of $ u_n$ in
$W^{1,q}(\omega')$, $\forall \omega' \subset \subset \Omega$,  imply
$$
u_n \longrightarrow u \quad \mbox{in}\quad W^{1,q}(\omega),  \quad \forall
1\leq q<\frac{N}{N-1}, \quad \forall  \omega \subset \subset \Omega,
$$
and, thanks to  \rife{a2}, we also have that
\begin{equation}\label{convergea}
a(x,u_n, \D u_n )\longrightarrow a(x,u,\D u) \quad \mbox{in}\quad L^1
(\omega)^N , \quad \forall \omega \subset \subset \Omega.
\end{equation}

Now we can pass to the limit in the distributional formulation: indeed choosing
any $\phi \in C^{\infty}_c (\Omega)$ in \rife{wap} we have
$$
\intO a(x,u_n, \D u_n )\cdot \D \phi   \, + \intO B(x,u_n)\phi   \,   =\intO
F_n \phi \,  .
$$
Using \rife{convergea} we deduce that
$$
\dys \lim_{n\to +\infty}
\int_{\mbox{\tiny supp } \phi }  a(x,u_n, \D u_n )\cdot \D \phi  \, =   \int_{\mbox{\tiny supp } \phi }  a(x,u , \D u )\cdot \D \phi
\, .
$$
Moreover, by the strong convergence of $\{B(x,u_n)\}$ and $\{F_n \}$ in
$L^1_{\rm loc} (\Omega)$, we deduce that
$$
\dys \lim_{n\to +\infty}
\int_{\mbox{\tiny supp } \phi } F_n\, \phi  \, =  \int_{\mbox{\tiny supp } \phi }  F\,  \phi  \,
$$
and
$$
\dys \lim_{n\to +\infty}
 \int_{\mbox{\tiny supp } \phi}  B(x,u_n)  \phi  \, =  \int_{\mbox{\tiny supp } \phi }  B(x,u ) \phi  \,
$$
and this concludes the proof.
\qed

\end{document}